\pdfoutput=1
\RequirePackage{silence}
\documentclass[a4paper,11pt]{amsart}
\usepackage[hmarginratio={1:1},vmarginratio={1:1},lmargin=60.0pt,tmargin=60.0pt]{geometry}

\allowdisplaybreaks

\usepackage[numbers]{natbib}
\usepackage[utf8]{inputenc}

\usepackage{latexsym,exscale,mathtools,textcomp}
\usepackage{amssymb,amsmath,amsthm,amsfonts,mathrsfs,enumitem}
\usepackage{blkarray}
\usepackage[table]{xcolor}
\usepackage{graphicx}
\usepackage{mathtools}
\usepackage{booktabs}

\usepackage{xparse}

\usepackage{dynkin-diagrams}

\setlist[enumerate]{itemsep=0.15cm,label=\emph{\upshape(\alph*)}}
\setlist[enumerate,2]{itemsep=0.15cm,label=\emph{\upshape(\roman*)}}

\usepackage{array}
\newcolumntype{C}{>{$}c<{$}}

\definecolor{mygray}{gray}{0.6}
\definecolor{mygraydark}{gray}{0.4}
\definecolor{mygraylight}{gray}{0.85}
\definecolor{spinach}{RGB}{46,139,87}
\definecolor{tomato}{RGB}{255,99,71}
\definecolor{orchid}{RGB}{143,40,194}
\definecolor{neon}{RGB}{77,77,255}
\definecolor{pumpkin}{RGB}{224,180,80}
\definecolor{citron}{RGB}{190,180,90}

\definecolor{lava}{RGB}{207,16,32}
\definecolor{cream}{RGB}{255,253,208}
\definecolor{verdigris}{RGB}{67,179,174}
\definecolor{Black}{RGB}{0,0,0}
\definecolor{mydarkblue}{RGB}{10,10,170}
\definecolor{darkspinach}{RGB}{20,70,20}
\definecolor{darktomato}{RGB}{155,40,30}
\definecolor{darkorchid}{RGB}{50,10,100}
\definecolor{darklava}{RGB}{150,8,16}

\usepackage{todonotes}

\let\emph\relax
\DeclareTextFontCommand{\emph}{\bfseries\em}

\newcommand{\placeholder}{{}_{-}}

\renewcommand{\dots}{\text{...}}

\let\<=\langle
\let\>=\rangle

\renewcommand{\dots}{\text{...}}

\DeclarePairedDelimiterX{\set}[1]{\{}{\}}{\setargs{#1}}
\NewDocumentCommand{\setargs}{>{\SplitArgument{1}{|}}m}{\setargsaux#1}
\NewDocumentCommand{\setargsaux}{mm}
{\IfNoValueTF{#2}{#1} {#1\,\delimsize|\,\mathopen{}#2}}

\newcommand{\Z}{\mathbb{Z}}

\usepackage[all]{xy}
\usepackage{tikz}
\usetikzlibrary{cd}
\usetikzlibrary{decorations}
\usetikzlibrary{decorations.markings}
\usetikzlibrary{decorations.pathreplacing}
\usetikzlibrary{decorations.pathmorphing}
\usetikzlibrary{arrows.meta,shapes,positioning,matrix,calc}
\usetikzlibrary{shapes.callouts}
\usetikzlibrary{shapes.arrows}
\usetikzlibrary{tqft}
\usetikzlibrary{scopes,intersections}

\tikzset{
anchorbase/.style={baseline={([yshift=#1]current bounding box.center)}},
anchorbase/.default={-0.5ex},
tinynodes/.style={font=\tiny,text height=0.25ex,text depth=0.05ex},
smallnodes/.style={font=\scriptsize,text height=0.75ex,text depth=0.15ex},
mor/.style={line width=0.75,color=black,fill=cream},
mor2/.style={line width=0.75,color=black,fill=tomato},
mor3/.style={line width=0.75,color=black,fill=spinach},
usual/.style={line width=1.2,color=black},
crossline/.style={preaction={draw=white,line width=5.0pt,-},preaction={draw=black,line width=0.9pt,-}},
hol/.style = {
decoration={markings,
post length=0.25mm,
pre length=0.25mm,
mark=at position #1 with {\node[circle,radius=0.15cm,inner sep=-1.2pt,draw,color=black,fill=white]{};}
},
postaction={decorate}
},
mob/.style = {
decoration={markings,
post length=0.25mm,
pre length=0.25mm,
mark=at position #1 with {\node[circle,radius=0.15cm,inner sep=-1.2pt,draw, color=tomato,fill=tomato]{};}
},
postaction={decorate}
},
dot/.style = {
decoration={markings,
post length=0.25mm,
pre length=0.25mm,
mark=at position #1 with {\node[circle,radius=0.15cm,inner sep=-1.2pt,color=black,fill=black]{};}
},
postaction={decorate}
},
dot/.default=1,
}
\tikzstyle directed=[postaction={decorate,decoration={markings,
mark=at position #1 with {\arrow[line width=0.25mm, black]{>}}}}]

\usepackage{aliascnt,etoolbox}
\def\NewTheorem#1{%
\newaliascnt{#1}{equation}%
\newtheorem{#1}[#1]{#1}%
\aliascntresetthe{#1}%
\expandafter\def\csname #1autorefname\endcsname{#1}%
}
\def\equationautorefname~#1\null{(#1)\null}

\numberwithin{equation}{subsection}

\NewTheorem{Proposition}
\NewTheorem{Theorem}
\NewTheorem{Corollary}
\AtEndEnvironment{Corollary}{\null\hfill$\square$}%
\NewTheorem{Lemma}
\theoremstyle{definition}
\NewTheorem{Definition}
\AtEndEnvironment{Definition}{\null\hfill$\Diamond$}%
\NewTheorem{Notation}
\AtEndEnvironment{Notation}{\null\hfill$\Diamond$}%
\NewTheorem{Example}
\AtEndEnvironment{Example}{\null\hfill$\Diamond$}%
\theoremstyle{remark}
\NewTheorem{Remark}
\AtEndEnvironment{Remark}{\null\hfill$\Diamond$}%

\usepackage[hypertexnames=false]{hyperref}
\usepackage{bookmark}
\hypersetup{
pdftoolbar=true,
pdfmenubar=true,
pdffitwindow=false,
pdfstartview={FitH},
pdftitle={RL unknotter, hard unknots and unknotting number},
pdfauthor={Anne Dranowski, Yura Kabkov and Daniel Tubbenhauer},
pdfsubject={},
pdfcreator={Anne Dranowski, Yura Kabkov and Daniel Tubbenhauer},
pdfproducer={Anne Dranowski, Yura Kabkov and Daniel Tubbenhauer},
pdfkeywords={},
pdfnewwindow=true,
colorlinks=true,
linkcolor=mydarkblue,
citecolor=teal,
filecolor=magenta,
urlcolor=orchid,
linkbordercolor=lava,
citebordercolor=teal,
urlbordercolor=orchid,
linktocpage=true
}

\def\makeautorefname#1#2{\csdef{#1autorefname}{#2}}

\makeautorefname{section}{Section}%
\makeautorefname{subsection}{Section}%
\makeautorefname{subsubsection}{Section}%

\begin{document}
\title[RL unknotter, hard unknots and unknotting number]{RL unknotter, hard unknots and unknotting number}
\author[A. Dranowski, Y. Kabkov and D. Tubbenhauer]{Anne Dranowski, Yura Kabkov and Daniel Tubbenhauer}

\address{A.D.: annedranowski@gmail.com}

\address{Y.K.: Neapolis University Pafos, Computer Science and Artificial Intelligence Bachelor program, 2 Danais Avenue, 8042, Pafos, Cyprus}
\email{y.kabkov@nup.ac.cy}

\address{D.T.: The University of Sydney, School of Mathematics and Statistics F07, Office Carslaw 827, NSW 2006, Australia, \href{http://www.dtubbenhauer.com}{www.dtubbenhauer.com}, \href{https://orcid.org/0000-0001-7265-5047}{ORCID 0000-0001-7265-5047}}
\email{daniel.tubbenhauer@sydney.edu.au}

\begin{abstract}
We train a reinforcement learning agent to simplify hard unknot diagrams and pair it with a learned crossing selector to search for short unknotting sequences. Four examples found by the agent lead to two infinite families of prime two-bridge and Montesinos knots of unknotting number three whose connected sums with the figure-eight knot also have unknotting number three, and fifteen new unknotting upper bounds were found on the set of prime knots up to 13 crossings.
\end{abstract}


\addtocontents{toc}{\protect\setcounter{tocdepth}{1}}

\maketitle

\tableofcontents

\section{Introduction}\label{sec:intro}

We train a reinforcement learning agent that treats knot simplification as search in a huge move graph and learns where to move next.

\subsection{Motivation}\label{sec:contrib}

A knot is a smooth embedding of $S^1$ in $S^3$ (``a closed rope''), and in practice it is often represented by a planar diagram.
Many basic questions in knot theory (deciding whether a diagram is the unknot, simplifying a diagram, searching for short unknotting sequences,\dots) are naturally diagrammatic: one applies local moves to a diagram and hopes to reach a simpler representative.
From a computational perspective this leads to a vast, highly branching move graph whose vertices are diagrams (or PD codes) and whose edges are elementary diagrammatic moves.
Even for moderately sized diagrams, exhaustive search is typically infeasible, and standard deterministic simplification heuristics can get stuck in deep local minima.

This paper explores the use of modern RL as a source of learned heuristics for navigating the move graph of knot diagrams.
Concretely, we formulate a Markov decision process in which states are planar diagram codes and actions are chosen diagrammatic moves (Reidemeister moves together with a small amount of auxiliary ``cleanup'').
We train a neural agent, the \emph{unknotter}, that (i) proposes promising moves and (ii) estimates how close a diagram is to being simplified.
The trained agent can then be used either as a policy (to generate simplification trajectories) or as a heuristic in a broader search procedure.

We also train a \emph{crossing selector}. Here the two agents have different jobs. We first inflate a diagram by isotopy. The selector ranks its crossings, and the unknotter simplifies the diagram after a selected crossing is changed. A sequence that ends in one crossing-free circle gives an upper bound for the \emph{unknotting number} $u(K)$, the least number of crossing changes needed to unknot $K$.

\subsection{Contributions}\label{sec:contribintro}

Our starting point is the work of Applebaum et al.\ on reinforcement learning and hard unknot diagrams \cite{ApplebaumBlackwellDaviesEdlichJuhaszLackenbyTomasevZheng-RLunknot}.
\begin{itemize}
\item \textit{Simplifying hard unknots.} Our unknotter learns both a move policy and a value estimate on planar diagram codes. On the released ``very hard'' evaluation file of \cite{ApplebaumBlackwellDaviesEdlichJuhaszLackenbyTomasevZheng-RLunknot} it finds a simplification to a circle in approximately $95\%$ of runs under the protocol of \autoref{sec:hard}.
\item \textit{Choosing a crossing.} The selector ranks crossings of an inflated diagram before we run the unknotter. On inflated validation diagrams it recovered the recorded upper bound on $47.57\%$ of visits, compared with $29.19\%$ for random selection, see \autoref{subsec:selector}.
\item \textit{Families with short unknotting sequences.} The agent found three-change sequences for three connected sums. Varying the length of their common twist region gives two-bridge knots $K_k$, starting with $9_{10}$, $12a1281$ and $13a2290$. We give a uniform three-change sequence for $4_1\#K_k$ and prove, for every even $k\geq0$, that
\[
 u(4_1\#K_k)=u(K_k)=3<4=u(4_1)+u(K_k).
\]
A fourth example gives a Montesinos family with the same equality.
\item \textit{Upper bounds for prime knots.} We give fifteen crossing-change witnesses leading to knots of unknotting number one. Together with the lower bounds discussed in \autoref{sec:upper bounds}, all fifteen source knots have unknotting number two.
\end{itemize}

Brittenham and Hermiller disproved additivity of unknotting number under connected sum and gave infinite families of counterexamples \cite{BrittenhamHermiller-unknotting-nonadditive}. Wang and Zhang gave a direct diagrammatic verification for the original example $7_1\#\overline{7_1}$ \cite{WangZhang-remark-counterexample-unknotting}. Brittenham and Hermiller subsequently found a partner of unknotting number three for the figure-eight, namely $9_{10}$ \cite[Theorem~1.1]{BrittenhamHermiller-unknotting-41-51}. Our first example $4_1\#K_0$ recovers their three-change bound and proves it sharp. The even members of our family keep the figure-eight fixed and give infinitely many distinct prime partners, with both the partner and the sum of unknotting number three. Dai, Hayman, Juh\'asz and Morellato give another route for $4_1\#9_{10}$ via the fifteen-crossing knot $15n4005$ \cite[Section~5.5]{DaiHaymanJuhaszMorellato-links}. Shewell Brockway gave a different family of two-bridge knots of unknotting number three containing $9_{10}$ up to mirror image \cite[Theorem~1]{ShewellBrockway-pretzel}. To our knowledge, the uniform result for the sums $4_1\#K_k$ is new.

\subsection{How to read this paper}\label{sec:howtoread}

\autoref{sec:background} fixes the diagram and move conventions. \autoref{sec:unknotter} describes the unknotter, and \autoref{sec:hard} reports its simplification benchmarks. \autoref{sec:crossing change} gives the crossing selection procedure, the three examples, and the proof of the family theorem. \autoref{sec:upper bounds} records the witnesses for prime knots. Code and data for the experiments are available at \cite{DKT}. The \href{https://github.com/dtubbenhauer/upperbounds-reworked/tree/main/family/certificates}{family certificates and their verification scripts} are on GitHub.

\noindent\textbf{Acknowledgments.}
We thank Seong-Jin Lee and Chuck Livingston for spotting mistakes in the previous version of this file.
DT subscribes to ``The game of life insists on its own permanence, independent of its players.'' and acknowledges support from ARC Future Fellowship FT230100489.

\noindent\textbf{AI declaration.} This paper was written in late 2025, when LLMs (at least the ones we used) were not yet capable of writing mathematics papers on their own. Unlikely as it may seem, this paper, including its ideas and proofs, is entirely human-made. AI was used solely for proofreading, phrasing, and coding, as was standard practice at the end of 2025. The same practice was used for the revisions, where we additionally used the mid-to-late-2026 models to check the arguments and certificates.

\section{Background and notation}\label{sec:background}

We assume the reader is comfortable with basic knot theory (diagrams, Reidemeister moves, crossing number), see e.g. \cite{Ad-knots} for background.
We nonetheless fix notation and briefly explain why the computational tasks we consider are intrinsically hard.

\subsection{Diagrams, planar diagram codes, and moves}\label{subsec:pd-0}

A \emph{knot diagram} is the image of a generic immersion $S^1\looparrowright S^2$ with crossing data (over/under information) at each double point. Links and link diagrams are multicomponent versions of knots and knot diagrams, and we will use them interchangeably.
Two diagrams represent the same knot if and only if they are related by a finite sequence of Reidemeister moves. The Reidemeister moves R0 (isotopy, often not specified), R1, R2, R3 are:
\begin{gather*}
\includegraphics[height=2.9cm]{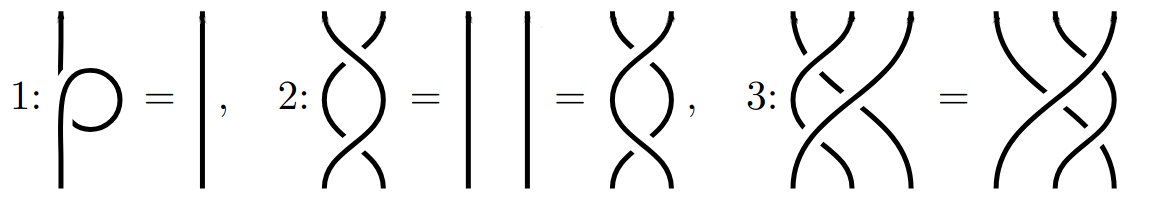}
\end{gather*}
So:
\begin{enumerate}
\item R0 is isotopy, and should be thought of as a cleaning operation.
\item R1 and R2 reduce or increase the number of crossings. We think of these as removing or adding cards to a deck.
\item R3 does not change the number of crossings. We think of this as a shuffle move.
\end{enumerate}
A crucial idea is \emph{increase-shuffle}, which increases the number of crossings using R1 and R2, and then shuffle them in using R3. For example:
\begin{gather*}
\includegraphics[height=5cm]{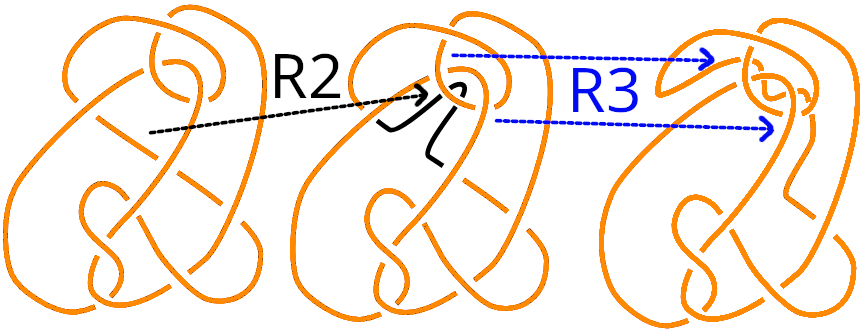}
\end{gather*}
In the picture on the right, the crossings created by the middle R2 move have been shuffled in.

\begin{Remark}
In general, we think of the unknotting as a game where the options are to add or remove cards (via R1 and R2), or to shuffle the deck (via R3).
\begin{gather*}
\includegraphics[height=5cm]{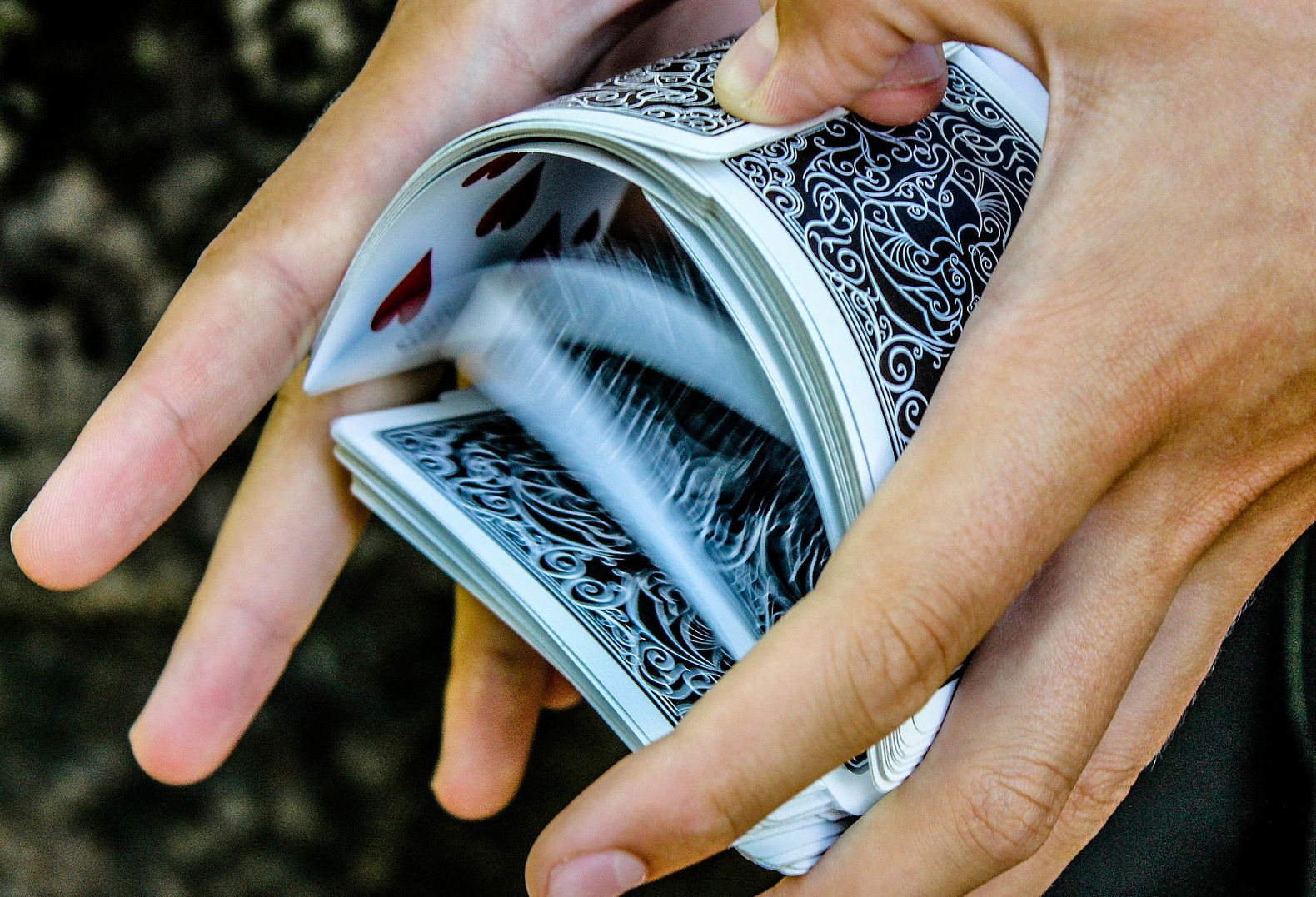}
\end{gather*}
This is an important (but certainly not new) idea to keep in mind.
\end{Remark}

In the implementation we represent a diagram by a \emph{planar diagram (PD) code}. Concretely, a PD code is a list of quadruples
\[
\texttt{PD} = \bigl[ [a_1,b_1,c_1,d_1],\dots,[a_n,b_n,c_n,d_n]\bigr],
\]
encoding the cyclic order of the four half-edges incident to each crossing, together with the over/under structure. In $[a,b,c,d]$, the opposite pair $a,c$ is the underpassing strand and $b,d$ is the overpassing strand, with entries listed counterclockwise. For example:
\begin{gather*}
\includegraphics[height=5.5cm]{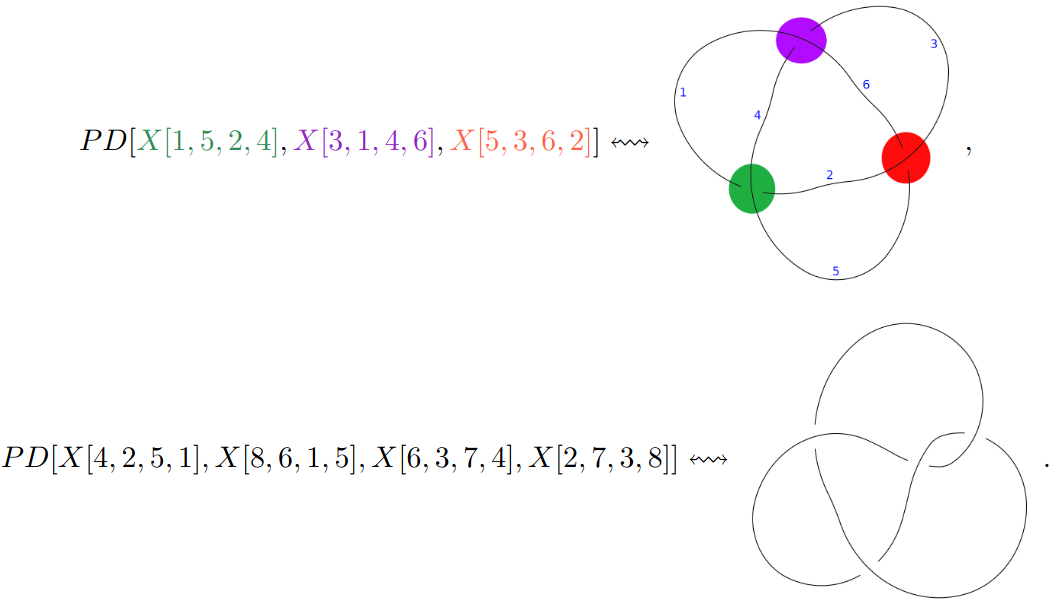}
\end{gather*}
(Sometimes PD codes are written $PD[X[a,b,c,d],\dots]$, but we usually omit the $PD$ and the $X$.)

\begin{Remark}
We work with diagrammatic encodings compatible with \texttt{spherogram} (as shipped with \texttt{SnapPy} \cite{CullerDunfieldWeeks-SnapPy}).
Throughout, a \emph{diagram} means a \emph{planar diagram} (PD) encoding of an oriented link diagram. We also use \emph{Dowker--Thistlethwaite (DT) codes} (in some of the extra material, not in this file) but they are less important for our purpose, so we will not recall them. We routinely convert between PD codes and other standard encodings whenever available.
\end{Remark}

Our move set consists of diagrammatic simplifications (Reidemeister 1/2/3 and standard ``cleanup'' and isotopy routines) together with a small amount of randomization (shuffling local R3 choices).
This is intentionally conservative: we want a move set that is uncontroversial for knot theorists and easy to verify independently.

\subsection{Why simplification is hard}\label{subsec:hardness}

Even for the unknot, simplification can be deceptive: many diagrams admit no immediate simplifying Reidemeister 1/2 moves, and a sequence of Reidemeister 3 moves may be required before any crossing reduction becomes available. For example:
\begin{gather*}
\includegraphics[height=5.5cm]{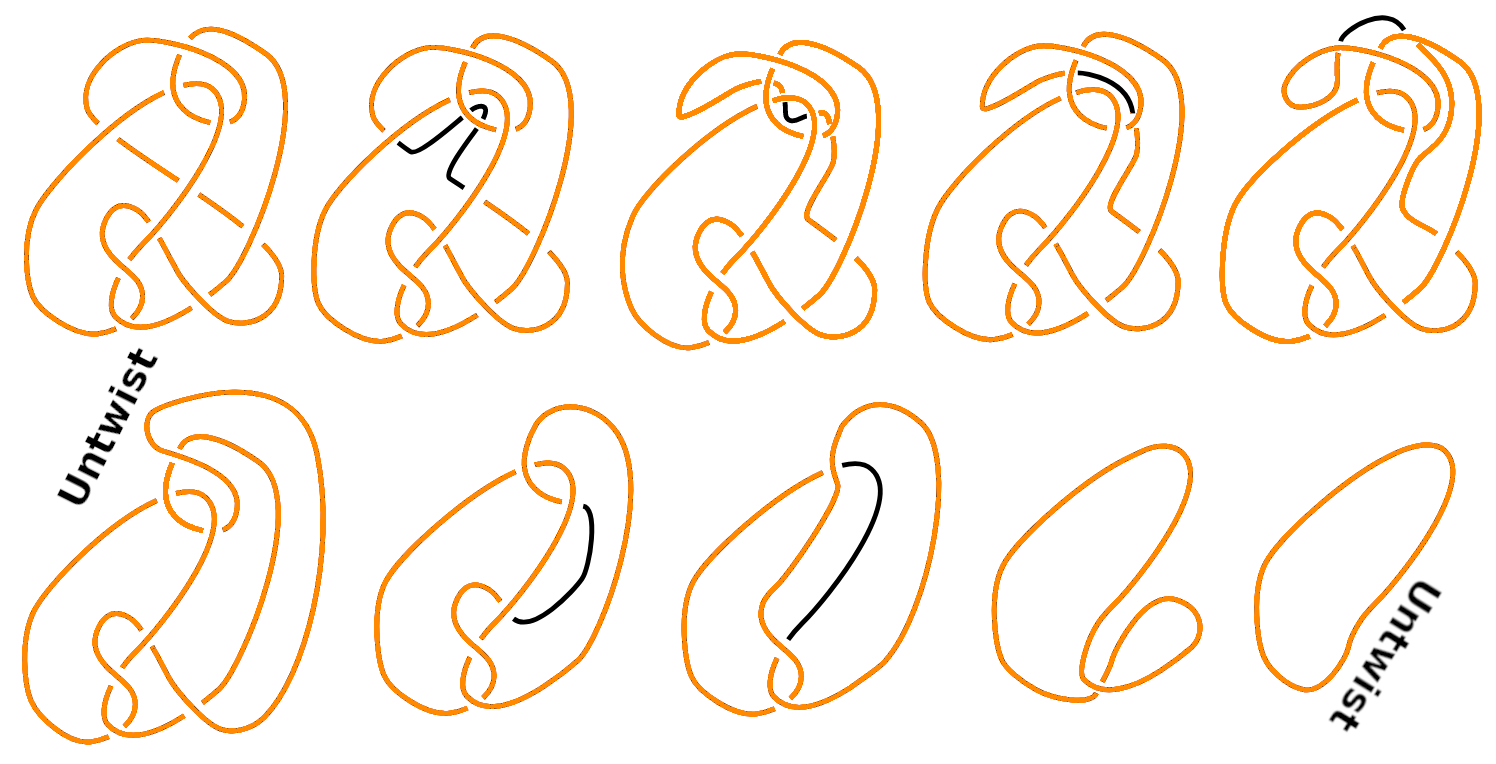}
\end{gather*}
(This is a famous ``local minimum'' example; it is taken from \cite{MR3193721} and then modified.)
Some unknot diagrams require an increase in crossings before any reduction becomes possible; see \autoref{subsec:hard-defs}. This is stronger than merely requiring a sequence of R3 moves.

From a computational viewpoint, this creates two intertwined difficulties:
\begin{itemize}
\item \textit{Sparse progress:} if we measure progress by the crossing number of the current diagram, then many ``correct'' steps (notably R3 moves) do not change this measure at all.
\item \textit{Local minima and ``dead ends'':} greedy strategies that only accept crossing reducing moves can (and will!) get stuck, while strategies that occasionally allow temporary increases in crossing number face an enormous branching factor.
\end{itemize}
This motivates the following guiding principle: we should allow the algorithm to explore, but we must also equip it with a mechanism to undo unproductive exploratory detours.
In our setup, \texttt{Link.backtrack} randomly applies crossing-increasing Reidemeister moves to escape a local trap. 
	
\begin{Remark}
The name is inherited from \texttt{spherogram}; this operation does not reverse the recorded move history.
\end{Remark}

Backtracking is not a new topological move; it is a search control device addressing the combinatorial explosion inherent in the Reidemeister search.

\subsection{A minimal RL primer}\label{subsec:rlprimer}

Reinforcement learning (RL) is a framework for learning \emph{policies} (decision rules) for sequential decision making problems, see e.g. \cite{SuttonBarto-RL} for background. (We assume the reader has background in RL and machine learning techniques, but not too much is actually needed to follow the exposition.)
An RL problem is typically formalized as a Markov decision process (MDP), consisting of:
\begin{enumerate}
\item a set of \emph{states} $s$,
\item a set of \emph{actions} $a$ available in each state,
\item a transition rule $s\mapsto s'$ when an action is applied, and
\item a \emph{reward} signal $r(s,a,s')$.
\end{enumerate}
The agent repeatedly observes a state (here: a diagram), chooses an action (here: a move or a macro step simplification mode), and receives reward.
The goal is to learn a policy maximizing expected cumulative reward (often discounted over time).
In our context, rewards are designed to correlate with ``getting simpler'': e.g. reducing the crossing number, reaching the unknot, or reaching a diagram that admits further reductions. For example (but simplified):
\begin{gather*}
\raisebox{-2.5cm}{\includegraphics[height=5cm]{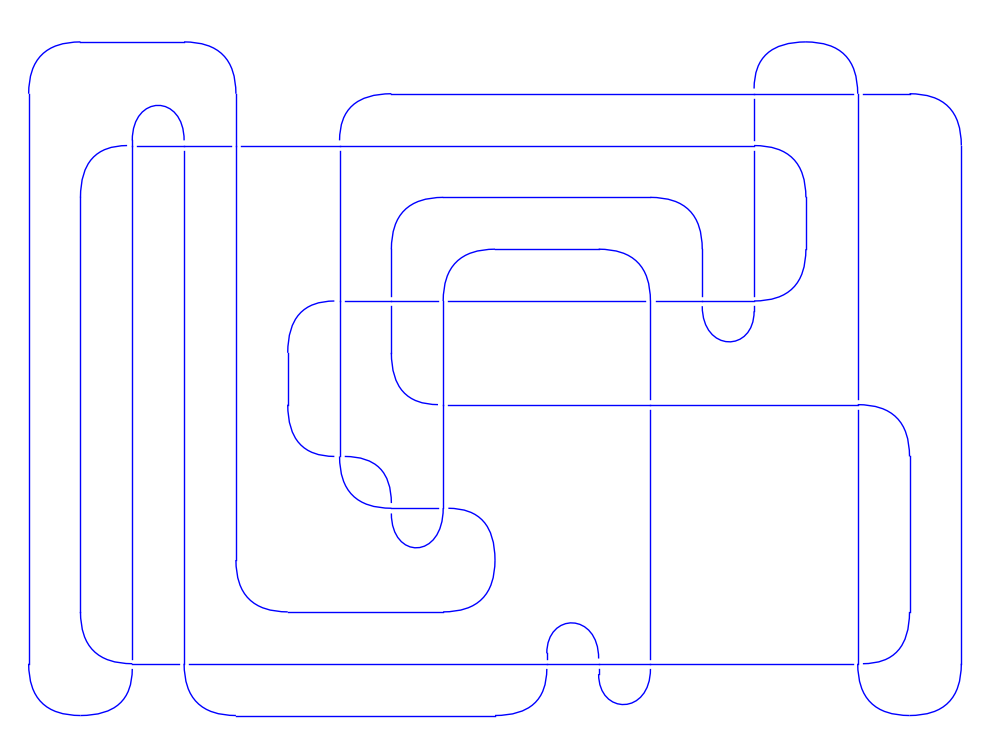}}\to
\raisebox{-2.5cm}{\includegraphics[height=5cm]{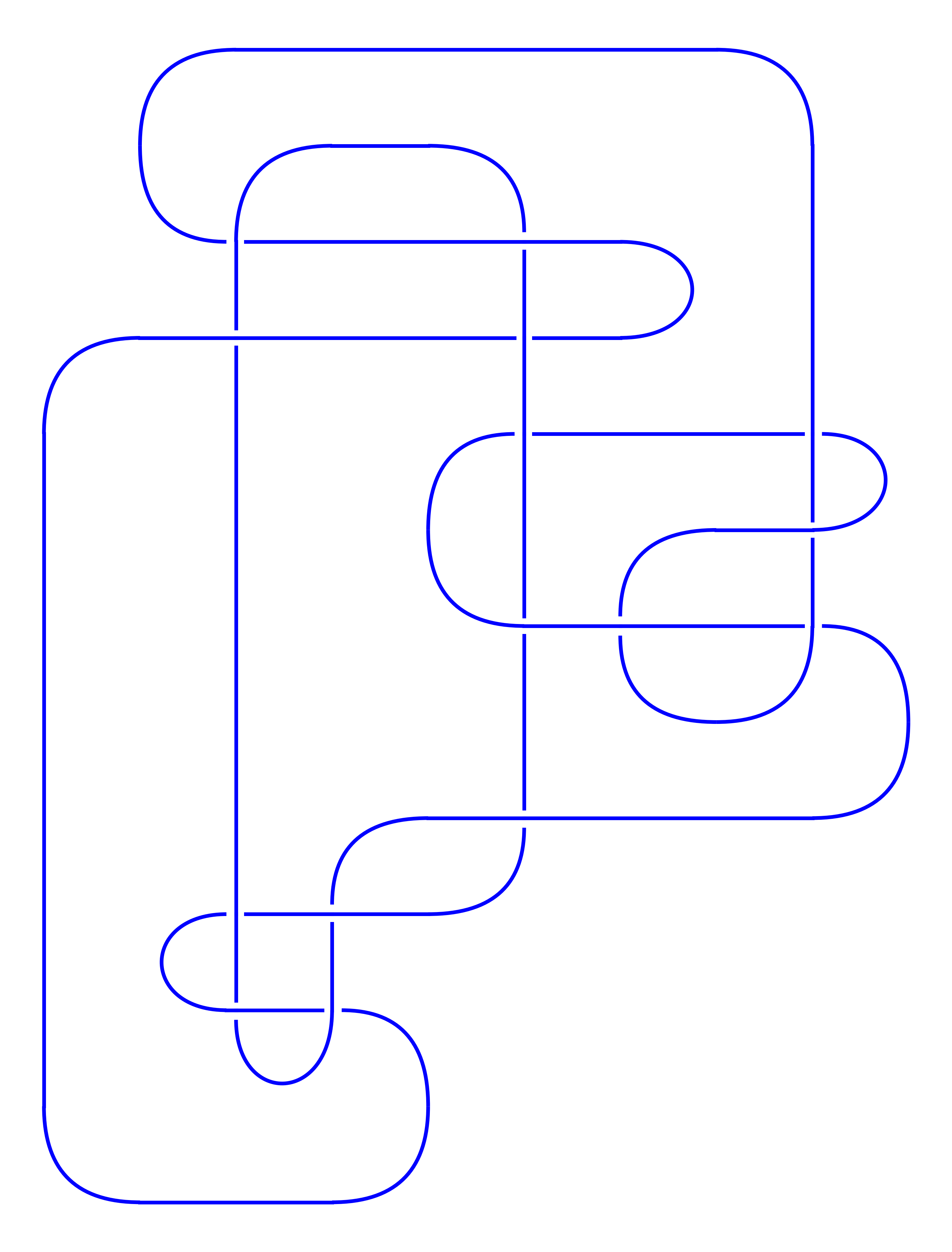}}\to\text{reward.}
\end{gather*}
For our purposes, the key point is conceptual rather than philosophical:
RL provides a mechanism to learn heuristics in large, irregular search spaces where hand-crafted greedy rules fail, especially when progress signals are sparse and the correct strategy sometimes requires non-monotone exploration (hence the importance of backtracking).
We will not attempt a general survey of RL; we only describe the specific Markov decision processes (MDP) and training setup used in this paper in \autoref{sec:unknotter}.

\subsection{Benchmarks used in this paper}\label{subsec:benchmarks}

Given a starting diagram $D_0$, we record the sampled trajectory under the above operations within a fixed step budget $T$.
We then define the \emph{empirical minimum crossing number}
\[
c_{\min}(D_0;T) \;:=\; \min_{0\leq t\leq\tau} c(D_t),\qquad \tau\leq T,
\]
This quantity depends on the allowed moves and on the budget $T$; it is not a certified invariant of the underlying link type. Here $D_0,\dots,D_\tau$ is the sampled trajectory, ending at termination or the step limit. Thus $c_{\min}$ also depends on the policy and random choices; no minimum over all reachable diagrams is computed.
In particular, $c_{\min}(D_0;T)$ is an upper bound on the true crossing number of the underlying link, and it is $0$ precisely when the simplification trajectory reaches a crossing-free diagram.

Our method, with details given below, is:
\begin{itemize}
\item Train an RL agent using macro actions built from R1, R2 and R3, with rewards for reducing the number of crossings. Planar isotopy R0 is implicit in the PD representation, rather than a separate policy action,
\item The agent outputs $c_{\min}(D_0;T)$.
\end{itemize}

\begin{Remark}
The code can also output additional data like the steps used or the minimal diagram, but that is not relevant for the abstract discussion below.
\end{Remark}

We will apply our method to three families of instances:
\begin{itemize}
\item the ``hard'' and ``very hard'' unknot diagrams introduced in \cite{ApplebaumBlackwellDaviesEdlichJuhaszLackenbyTomasevZheng-RLunknot}, and
\item connected sums of the figure-eight with the two-bridge knots in \autoref{thm:figure-eight-family}, and
\item prime knots after a crossing change.
\end{itemize}
All sources provide diagrammatic instances that are difficult for naive greedy simplifiers because many local moves temporarily increase $c(D)$ or require nonlocal rearrangements before a crossing reduction becomes available.

\section{The reinforcement learning unknotter}\label{sec:unknotter}

We describe the general unknotter in \path{notebooks/uncrossing.ipynb} \cite{DKT}. The benchmark for hard unknots uses the variant described in \autoref{subsec:hard-protocol}.
The agent is trained on a \texttt{gymnasium} environment whose states are diagrams and whose actions are \emph{macro-actions} that call \texttt{spherogram}'s simplification routines.
We train using PPO as implemented in \texttt{stable-baselines3}.

\begin{Remark}
We call the model below the \emph{unknotter}.
The unknotter is available here: \cite{DKT}.
\end{Remark}

\subsection{Planar-diagram encodings}\label{subsec:pd}

Given a diagram $D$, we write $c(D)$ for its number of crossings.
Our code accepts the following PD/DT input formats (and we use these formats in all datasets released with this paper):
\begin{itemize}
\item \textbf{DT code:} a string of the form \texttt{DT: [\dots]} which is passed directly to \texttt{spherogram.Link}.
\item \textbf{PD list-of-quads:} a string representing a list of quadruples
\[
[[a_1,b_1,c_1,d_1],\ldots,[a_n,b_n,c_n,d_n]],
\]
parsed as JSON (or, as a fallback, as a Python literal) and then passed to \texttt{spherogram.Link}.
\item \textbf{PD in \texttt{X[\dots]} blocks:} a string containing one or more blocks \texttt{X[a,b,c,d]}; these blocks are extracted and converted into a list-of-quads.
\item \textbf{Simple JSON wrappers:} if a line is valid JSON containing a field \texttt{pd}/\texttt{dt} (or variants), we unwrap and parse that field.
\end{itemize}
In all cases we apply a strict parser (rejecting lines that are not clearly PD/DT encodings) and keep only lines that successfully instantiate a \texttt{Link} object.

\subsection{Environment: state, actions, and termination}\label{subsec:env}

To summarize, the environment state includes a \texttt{spherogram.Link} object $D$ obtained as in \autoref{subsec:pd}, together with the step counter, blocked modes and previous action information.
Episodes start by sampling a diagram uniformly at random from a fixed list of input PD strings (a mixture of hard/very-hard instances and random diagrams; see \autoref{subsec:data}).

Rather than just feeding the full PD graph to the network, the current implementation uses a compact feature vector
$o_t\in \mathbb{R}^6$
with the following entries:
\begin{enumerate}
\item $c(D)$, the current number of crossings,
\item the number of link components (when available from \texttt{spherogram}),
\item the current step counter,
\item a binary feature indicating whether a single \texttt{basic} simplify call would strictly reduce crossings,
\item a binary feature indicating whether the previous action reduced the crossing number,
\item a constant bias feature.
\end{enumerate}

The action space (macro-actions) is as follows.
An action is a pair $(m,\kappa)$ in a \texttt{MultiDiscrete} space, where
\[
m\in\{0,1,2,3\},\qquad \kappa\in\{0,1,\dots,\kappa_{\max}\},\;\;\kappa_{\max}=8.
\]
(In other words, the model can choose its action, details up next, and how often it wants to apply them.)
The mode $m$ selects one of four macro-actions (details on the moves follow in \autoref{subsec:moves}):
\begin{itemize}
\item $m=0$: \texttt{simplify(mode="basic")} (this tries to reduce the number of crossings using all R1 and R2 instances it can find);
\item $m=1$: \texttt{simplify(mode="level", type\_III\_limit=$\kappa\vee 1$)} (R3 moves with R1 and R2 reductions when available);
\item $m=2$: \texttt{simplify(mode="pickup", type\_III\_limit=$\kappa\vee 1$)} (reroutes an entirely overpassing or underpassing strand to minimize its crossings);
\item $m=3$: \texttt{backtrack(steps=$\kappa\vee 1$, prob\_type\_1=0.35, prob\_type\_2=0.65)} (this increases the number of crossings by adding R1 and R2 moves randomly) followed by a small type~3 shuffle.
\end{itemize}
(Here $\kappa\vee1:=\max\{\kappa,1\}$.)
The backtrack macro action can be disabled; in this paper we keep it enabled.
An episode terminates when either
(i) $c(D)=0$ (a crossing-free diagram, representing the unknot when the input has one component), or
(ii) a fixed step budget $T=500$ is reached.

\begin{Remark}
The values $\kappa_{\max}=8$ and $T=500$ are the ones that worked well. They can be changed if the user wants to.
\end{Remark}

\subsection{Move set implemented in \texttt{spherogram}}\label{subsec:moves}

The core simplification primitive is \texttt{spherogram}'s built-in diagram simplifier
\[
D \mapsto \texttt{simplify}(D;\,\texttt{mode}=\text{basic/level/pickup}),
\]
which performs combinations of Reidemeister moves (and related local simplifications) according to the chosen mode. These routines apply available reductions heuristically; the resulting diagram need not be unique and can depend on move order and random choices.
In addition, we use \texttt{Link.backtrack} as a stochastic ``escape'' operator that attempts to undo local traps by exploring alternative simplification paths.

\subsection{Reward shaping and ``blocking''}\label{subsec:reward}
Let $c_{\text{before}}$ and $c_{\text{after}}$ be the crossing numbers before and after applying the chosen macro-action, and set
\[
\Delta := c_{\text{before}}-c_{\text{after}}.
\]
The per-step reward is a dense shaping term plus a small step cost:
\[
r \;=\; w_{\Delta}\,\Delta \;-\; w_{\uparrow}\,\max\{0,-\Delta\}\;-\; w_{\mathrm{pot}}\,c_{\text{after}}\;-\;\lambda,
\]
with parameters
\[
(w_{\Delta},w_{\uparrow},w_{\mathrm{pot}},\lambda)=(1.0,\,0.5,\,0.02,\,0.05),
\]
and an additional terminal bonus of $+10$ upon reaching $c(D)=0$.

To reduce wasted exploration on obviously unproductive modes, we implement a simple \emph{blocking} heuristic:
if a non-backtrack mode $m\in\{0,1,2\}$ increases crossings (i.e.\ $\Delta<0$), we temporarily mark that mode as blocked, and future requests for that mode are cyclically remapped to the next unblocked mode.
Any successful crossing reduction ($\Delta>0$) resets the blocked set; additionally, using backtrack also resets the blocked set.

\begin{Remark}
Implementing the blocking feature was crucial. Without it, the unknotter failed to do its designed task properly.
\end{Remark}

\subsection{Policy and training}\label{subsec:training}

We train a Proximal Policy Optimization (PPO) agent with an MLP policy network (\texttt{MlpPolicy} in \texttt{stable-baselines3}).
Unless stated otherwise, the hyperparameters are as follows:
\begin{itemize}
\item learning rate: linear decay $3\cdot 10^{-4}\cdot p$ (with training progress $p\in[0,1]$),
\item rollout length $n_{\mathrm{steps}}=2048$, batch size $256$, epochs per update $10$,
\item discount $\gamma=0.995$, GAE parameter $\lambda_{\mathrm{GAE}}=0.97$,
\item clip range $0.2$, entropy coefficient $0.01$, value-function coefficient $0.5$,
\item max gradient norm $0.5$.
\end{itemize}
Models are evaluated periodically during training using a one-episode evaluation callback, and we keep the best-performing checkpoint.

\begin{Remark}
These parameters are not set in stone; tweaking might give better results.
\end{Remark}

\subsection{Training data}\label{subsec:data}

Training episodes sample PD/DT strings from a mixture of sources.
In the current codebase we use:
\begin{itemize}
\item a large file of ``hard'' unknots and a small file of ``very hard'' unknots from \cite{ApplebaumBlackwellDaviesEdlichJuhaszLackenbyTomasevZheng-RLunknot}, stored as CSVs (first column = PD/DT string), and
\item auxiliary sets of random diagrams used for exploration and regularization.
\end{itemize}
All lines are filtered by the strict parser from \autoref{subsec:pd}, ensuring that every training instance can be instantiated as a \texttt{spherogram.Link}.

\begin{Remark}
Recognizing knots in photographs poses a different problem. See 
\cite{DranowskiKabkovTubbenhauer-knot-detection}.
\end{Remark}

\section{Unknotting the hard and ``very hard'' unknots}\label{sec:hard}

This section tests the unknotter on the hard unknot datasets introduced in \cite{ApplebaumBlackwellDaviesEdlichJuhaszLackenbyTomasevZheng-RLunknot}.
We first recall the terminology used there, then describe our evaluation protocol, and finally report the results.

\subsection{``Hard'' and ``very hard''}\label{subsec:hard-defs}

Following \cite{ApplebaumBlackwellDaviesEdlichJuhaszLackenbyTomasevZheng-RLunknot}, a diagram $D$ of the unknot is called \emph{hard} if, in any sequence of Reidemeister moves from $D$ to the trivial diagram, the crossing number must first increase before it decreases.
In particular, there is no monotone simplification of $D$ to the trivial diagram with respect to $c(\placeholder)$.
Hard unknot diagrams are of particular interest because they obstruct naive greedy simplification heuristics and may serve as stress tests for unknot detection procedures.

\begin{Remark}
Our task is diagram simplification. Difficulty of simplification does not imply difficulty of unknot recognition by other methods. \cite{CantarellaSchumacherShonkwiler-ReAPR} simplified the published hard unknot examples they tested by combining pass moves with geometric re-embedding.
\end{Remark}

In addition, \cite{ApplebaumBlackwellDaviesEdlichJuhaszLackenbyTomasevZheng-RLunknot} identifies a much smaller subset of ``really hard'' diagrams: these are hard unknot diagrams on which \texttt{SnapPy}'s built-in heuristic simplification (including its more sophisticated \texttt{global} heuristic) still fails to reach the trivial diagram even after repeated attempts; see \cite{ApplebaumBlackwellDaviesEdlichJuhaszLackenbyTomasevZheng-RLunknot}.
In this paper we will refer to these ``really hard'' instances as \emph{very hard} unknots. An example is:
\begin{gather*}
\includegraphics[height=5.5cm]{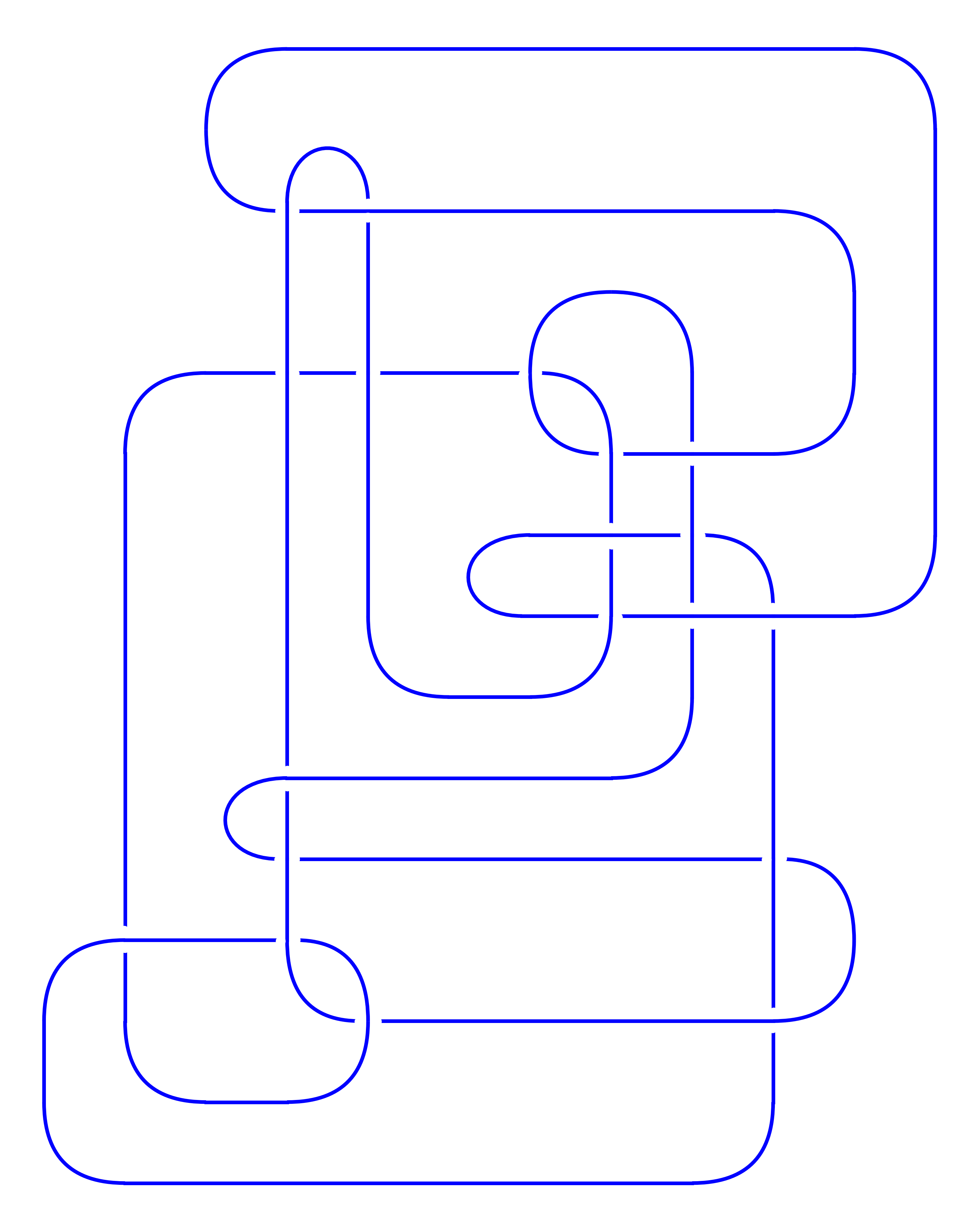}
\end{gather*}

\subsection{Experimental setup}\label{subsec:hard-protocol}

We work with the PD/DT formats described in \autoref{subsec:pd}.
The released \path{very_hard_unknots.csv} contains 385 distinct PD strings. Applebaum et al.\ report 382 diagrams in their paper \cite[Appendix~A]{ApplebaumBlackwellDaviesEdlichJuhaszLackenbyTomasevZheng-RLunknot}. Our counts use the 385 file entries. The training pool included this file, so the experiment measures repeated search on familiar inputs.
In all experiments we apply our strict parser and keep only strings that instantiate a \texttt{spherogram.Link} object (see \autoref{subsec:pd}).

The released benchmark notebook \path{hard_unknot.ipynb} \cite{DKT} sets the action cap to $16$ and the training horizon to $1000$ macro-steps. Its reward for the four effective modes, in the order of \autoref{subsec:env}, is $3,2,1,0$ respectively, less $0.05$ per step, with a bonus of $10$ at zero crossings. It blocks a non-backtrack mode whenever the crossing count fails to decrease. The recorded evaluation uses a deterministic policy with a budget of $T=500$ macro-steps.
Each step is one macro-action of the form \texttt{simplify(mode, cap)} (with \texttt{mode} in \{\texttt{basic}, \texttt{level}, \texttt{pickup}, \texttt{backtrack}\}, cf.\ \autoref{subsec:moves}), plus the small post-backtrack shuffle by pure type~3 moves (typically $k\le 2$) to avoid immediate cycling.

For a start diagram $D_0$, let $c_{\min}(D_0;T)$ be the empirical minimum crossing number encountered along the trajectory.
We record:
\begin{itemize}
\item \textit{Unknot success:} whether the trajectory reaches the trivial diagram, i.e.\ $c_{\min}(D_0;T)=0$.
\item \textit{Crossing reduction:} whether the agent achieves any reduction, i.e.\ $c_{\min}(D_0;T)<c(D_0)$.
\end{itemize}
The latter is a weaker diagnostic that is useful for sanity checks and ablations; the main target is actual unknotting.

Although the policy evaluation is deterministic, the environment contains stochasticity through: (i) \texttt{backtrack}, which is stochastic in \texttt{spherogram}, and (ii) our optional post-backtrack random type~3 shuffle.
Hence we repeat each evaluation $R$ times per instance with independent random seeds.

\subsection{Results on very hard unknots}\label{subsec:hard-results}

We focus on the ``very hard'' subset from \cite{ApplebaumBlackwellDaviesEdlichJuhaszLackenbyTomasevZheng-RLunknot}.
Empirically, these are precisely the instances where naive simplifiers and even repeated SnapPy \texttt{global} calls often stall, making them a robust stress test for our approach.

\begin{Proposition}
(Empirical.) On the $N=385$ entries of the released file, evaluated with a fixed budget of $T=500$ macro-steps and repeated $R=10$ times per instance, the unknotter achieves a mean per-run unknotting rate of $p=94.57\%$ with standard deviation $1.15\%$ across runs. In total, $3{,}641$ of the $3{,}850$ trials reached zero crossings.

A stronger way to summarize robustness is instance-wise: in these experiments, every instance was unknotted in \emph{at least one} of the $R=10$ runs (so there were no ``never resolved'' diagrams at this budget), $266$ instances were unknotted in \emph{all} $10$ runs, and the remaining $119$ instances were unknotted only \emph{intermittently} (solved in some but not all runs), indicating residual sensitivity to stochastic choices in \texttt{backtrack}/shuffle.
\end{Proposition}

\begin{proof}
Here is the summary table:
\begin{gather*}
\includegraphics[height=5.5cm]{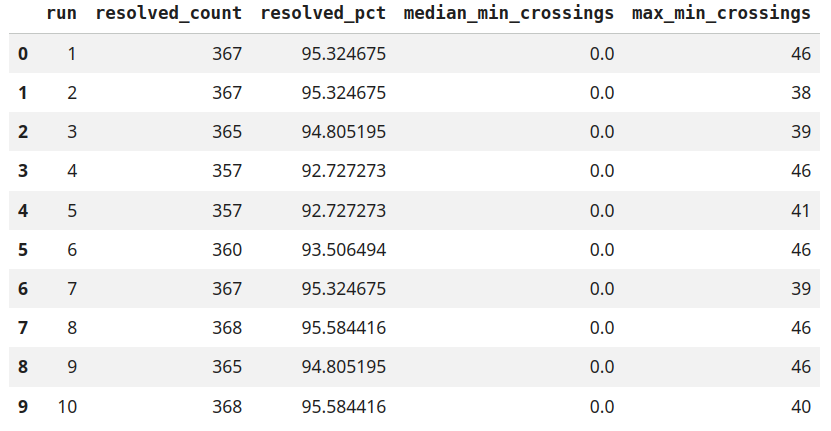}\\\includegraphics[height=5.5cm]{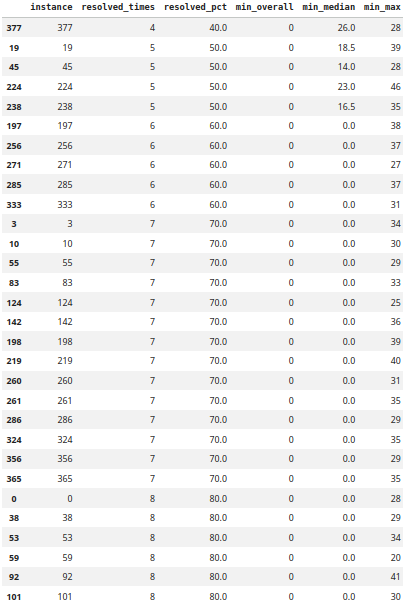}
\end{gather*}
The figures above summarize our evaluation on the $N=385$ entries in the benchmark file with a fixed budget of $T=500$ macro-steps, repeated for $R=10$ independent runs.
The first figure reports per-run aggregates: for each run it lists the number and percentage of instances that were fully resolved (i.e. reduced to a crossing-free diagram), together with the median of the minimum crossing number attained over instances in that run (which is $0$ throughout) and the worst minimum achieved among the remaining unresolved instances (typically in the range $38$ to $46$).
The second figure is instance-wise: it ranks diagrams by how often they were resolved across the $10$ runs and records, for each instance, the min/median/max of the best crossing number achieved across runs.
In particular, all instances are resolved at least once, but a noticeable subset is only intermittently resolved (e.g. $4/10$, $5/10$, \dots), indicating sensitivity to stochastic choices in the simplification/backtracking dynamics at fixed budget.
\end{proof}

The public notebook preserves the per-run counts. Replaying the experiment requires the checkpoint and episode helper used in the original session.

\subsection{Why backtracking matters}\label{subsec:hard-ablation}

One conceptual reason for introducing the \texttt{backtrack} action is that it directly targets the ``local trap'' phenomenon discussed in \autoref{subsec:hardness}:
the agent may spend many steps in a region where $c(D)$ does not decrease, and allowing temporary crossing increases can expose new simplification paths.
Backtracking provides a controlled way to abandon a nonproductive branch and try an alternative simplification path.

\section{Crossing selection and a family of figure-eight sums}\label{sec:crossing change}

Brittenham and Hermiller showed that three crossing changes suffice for $4_1\#9_{10}$, although unknotting the two factors separately takes four \cite{BrittenhamHermiller-unknotting-41-51}. Our agent recovered a three-change sequence for this sum and found two more, with $12a1281$ and $13a2290$ in place of $9_{10}$. The diagrams suggested varying the number of crossings in a single twist region. We describe the search first and then prove what persists as the number of twists grows.

\subsection{Choosing crossings on inflated diagrams}\label{subsec:selector}

A crossing change interchanges the overstrand and understrand at one crossing. Reidemeister moves between changes cost nothing in the unknotting number. This freedom matters in the search: a minimal diagram can hide crossings between the parts of a connected sum. We use the following procedure for a diagram $D$ of a knot $K$.
\begin{enumerate}
\item \textit{Inflate.} Insert crossings by Reidemeister moves and rearrange the diagram by isotopy, as in \autoref{subsec:moves}. The result is a new diagram of $K$.
\item \textit{Select and change.} The learned crossing selector ranks crossings of the inflated diagram. Spend the search budget on highly ranked crossings, with random choices for exploration. Each trial changes one crossing on a fresh copy of the diagram.
\item \textit{Simplify and check.} Apply diagram simplification and, when needed, the trained unknotter to the changed diagram. A reduction to one circle is an unknotting witness. A verified identification of another knot $L$ with a known upper bound $b$ gives $u(K)\leq 1+b$. Otherwise the simplified diagram can be used for a further crossing change. Record each change in the diagram on which it is made, together with the intervening isotopies.
\end{enumerate}
A crossing can be useful even when the changed diagram is still knotted. For example, in training a crossing ranker we write $U(K)$ for the upper bound in the training table. A crossing from $K$ to $J$ recovers that bound if $1+U(J)\leq U(K)$. In the \href{https://github.com/dtubbenhauer/upperbounds-reworked/blob/main/selector/earlier-experiments/smart_crossing_recovery_v2_colab.ipynb}{earlier recovery notebook} the input graph has one vertex per crossing. Its edges record the ports joined by each arc, so the ranker can distinguish the over- and underpassing strands independently of the arbitrary arc labels.

In the run recorded on 2~October 2026, we compared the two choices on the same $15{,}976$ visits to inflated validation diagrams. The highest-scoring crossing recovered the recorded bound on $7{,}600$ visits ($47.57\%$). Uniform random selection among the distinct flip outcomes recovered it on $4{,}663$ visits ($29.19\%$). These counts pool visits during training, as the selector weights changed. The \href{https://github.com/dtubbenhauer/upperbounds-reworked/tree/main/selector}{selector directory} contains the notebook and the recorded counts.

\subsection{The family}\label{subsec:family-statement}

In a two-bridge diagram the entries of a continued fraction count the signed twists in successive boxes. We use the convention $[a_1,\ldots,a_m]^+=a_1+1/[a_2,\ldots,a_m]^+$ and $[a_m]^+=a_m$, with $1/0=\infty$. Closing the rational tangle of fraction $p/q$ by joining the two upper ends and the two lower ends gives its \emph{numerator closure}, see e.g. \cite{KauffmanLambropoulou-rational}. Write $F=4_1$. For each integer $k\geq0$ let $K_k$ be the numerator closure with fraction
\begin{equation}\label{eq:family-fraction}
 [3,3,3,2,k]^+=\frac{76k+33}{23k+10}.
\end{equation}
Thus only the last twist box varies. Its zero-twist specialization gives $[3,3,3]^+=33/10$. The numerator is odd, so the closure is a knot.

\begin{Theorem}\label{thm:figure-eight-family}
For every $k\geq0$, three crossing changes suffice to unknot $K_k$ and three suffice to unknot $F\#K_k$. For every even $k\geq0$ the upper bounds are sharp:
\[
 u(F\#K_k)=u(K_k)=3<4=u(F)+u(K_k).
\]
The $K_k$ are pairwise distinct prime knots. The assertions hold for either mirror convention for $K_k$.
\end{Theorem}

The fraction in \eqref{eq:family-fraction} is reduced. Its numerator $76k+33$ is the determinant of $K_k$. These determinants distinguish the knots, and each is a nontrivial two-bridge knot, hence prime. For even $k\geq2$, the standard reduced alternating diagram for $[3,3,3,2,k]^+$ has $11+k$ crossings. Thus the partners of unknotting number three realize every odd crossing number from thirteen onward. The figure-eight is amphichiral and has unknotting number one, as seen by changing a crossing in its standard diagram. Mirroring preserves unknotting number. The crossing changes are proved in \autoref{subsec:uniform}, and \autoref{prop:family-lower} supplies the lower bounds. We first draw the three instances that suggested the conjecture.

\begin{Remark}
The agent found the unexpected unknotting sequences for $k=0,1,2$, which then made it, in hindsight, easy for us to guess the general pattern. The key step was mostly luck: identifying $9_{10}$ and the other two knots as the above rational tangles was a lucky strike, but then made the pattern obvious.
\end{Remark}

\subsection{The first three examples}\label{subsec:family-examples}

The agent found the cases $k=0,1,2$. To compare the witnesses we put them in a common diagrammatic form, varying one twist region in the Brittenham--Hermiller construction. In each row
of \autoref{fig:finite-family-crossings}, change the two circled crossings on
the left. An isotopy gives the diagram on the right. Change its circled
crossing and the diagram simplifies to one circle. We use knot table names
up to mirror.

\begin{figure}[htbp]
\centering
\begingroup
\setlength{\tabcolsep}{4pt}
\begin{tabular}{@{}c@{\quad}c@{\quad}c@{\quad}c@{}}
  & $4_1\#K_k$ & & after two changes \\[4pt]
  $k=0$ & \raisebox{-.46\height}{\includegraphics[width=.355\linewidth,height=4.25cm,keepaspectratio]{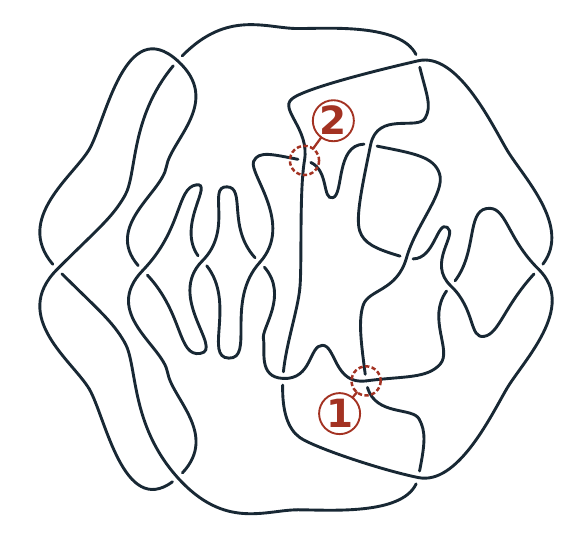}} & $\longrightarrow$ & \raisebox{-.46\height}{\includegraphics[width=.355\linewidth,height=4.25cm,keepaspectratio]{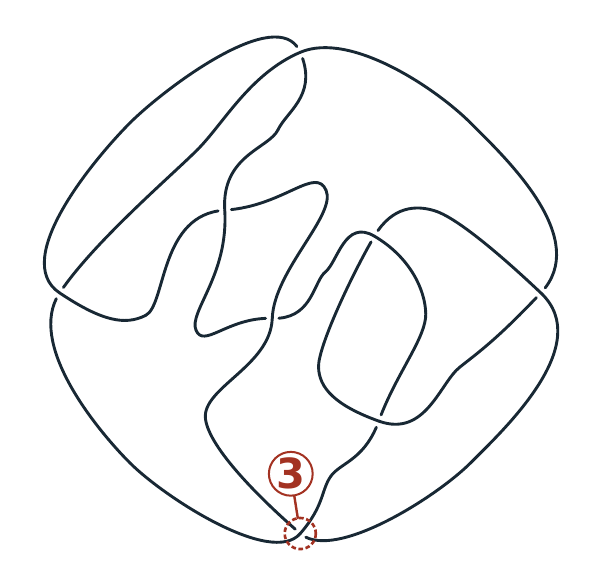}} \\[7pt]
  $k=1$ & \raisebox{-.46\height}{\includegraphics[width=.355\linewidth,height=4.25cm,keepaspectratio]{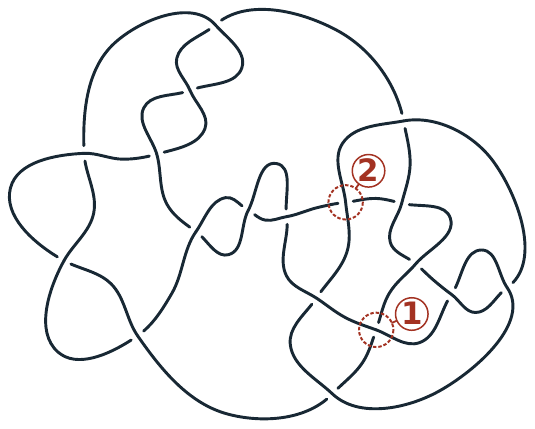}} & $\longrightarrow$ & \raisebox{-.46\height}{\includegraphics[width=.355\linewidth,height=4.25cm,keepaspectratio]{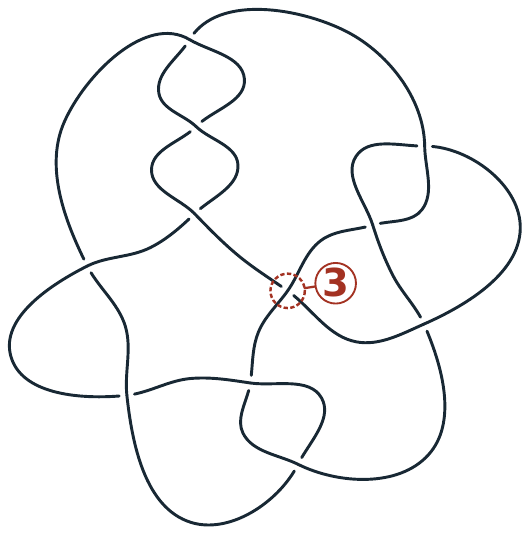}} \\[7pt]
  $k=2$ & \raisebox{-.46\height}{\includegraphics[width=.355\linewidth,height=4.25cm,keepaspectratio]{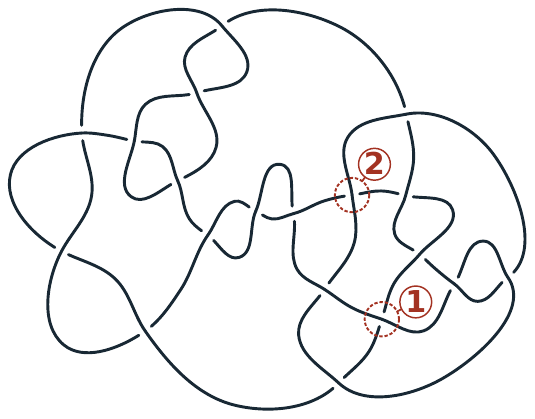}} & $\longrightarrow$ & \raisebox{-.46\height}{\includegraphics[width=.355\linewidth,height=4.25cm,keepaspectratio]{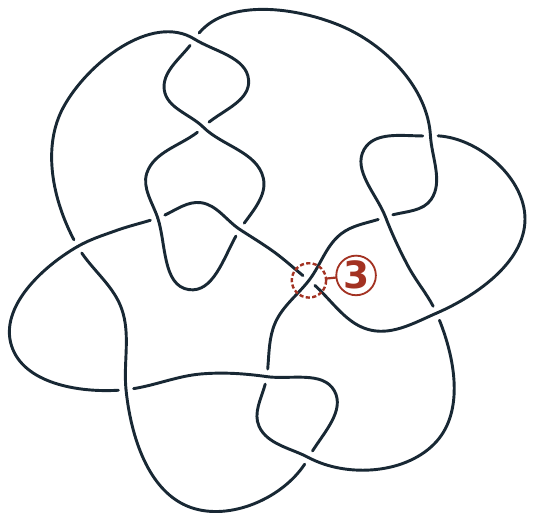}}
\end{tabular}
\endgroup
\caption{Three changes for each connected sum. The red numbers give the
order of the changes. Reading from top to bottom, the diagrams on the left
have $15,18,19$ crossings and those on the right have $8,11,12$.}
\label{fig:finite-family-crossings}
\end{figure}

\begin{center}
\begin{tabular}{ccccc}
\toprule
$k$ & $K_k$ & after two changes & $u(K_k)$ & $u(4_1\#K_k)$ \\
\midrule
0 & $9_{10}$ & $8_{14}$ & $3$ & $3$ \\
1 & $12a{1281}$ & $11a{12}$ & $[2,3]$ & $[2,3]$ \\
2 & $13a{2290}$ & $12a{28}$ & $3$ & $3$ \\
\bottomrule
\end{tabular}
\end{center}

Brittenham and Hermiller gave the three-change sequence through $8_{14}$
for $k=0$ \cite[Theorem~1.1]{BrittenhamHermiller-unknotting-41-51}.
For $k=1$, the fraction $[3,3,3,3]=109/33$ identifies the factor as $12a{1281}$, by e.g. \cite[Table~A1]{DiPrisa-strongly-invertible}. Lee lists
$12a{1281}$ among the knots with unknotting number two or three
\cite[Appendix~F]{Lee-BJ}. The sum is nontrivially composite, so it needs
at least two changes \cite{Scharlemann-prime}. These bounds leave open whether this odd example saves a change. For $k=2$ the fraction is $185/56$. Gebel and Prangley proved $u(13a{2290})=3$ \cite[Theorem~2]{GebelPrangley-alternating}. Lee also records this value \cite[Appendix~C]{Lee-BJ}.
Both even examples save exactly one change by \autoref{prop:family-lower}.

Here is how to read the marks in the PD codes. Number the rows of each
diagram from zero. Marks 1 and 2 are rows 1 and 5 of the left panel. Mark 3 is row 3 of the right panel. For example, when $k=2$ the three
marked quadruples change as shown here.
\[
\begin{aligned}
 (26,15,27,16)&\longmapsto(15,27,16,26),\\
 (29,10,30,11)&\longmapsto(10,30,11,29),\\
 (4,21,5,22)&\longmapsto(21,5,22,4).
\end{aligned}
\]
The first two changes take place in disjoint crossing balls. Five
Reidemeister~III moves and four strand isotopies then give the right-hand
panel in each case. After the third change, respectively $7$, $16$ and $17$
Reidemeister moves give a crossing-free circle. The starting knots can
also be checked directly. Three Reidemeister~III moves and one
Reidemeister~II move expose a separating circle meeting the diagram
twice, with $4_1$ on one side and the indicated rational knot on the other.

The GitHub directory \href{https://github.com/dtubbenhauer/upperbounds-reworked/tree/main/family/certificates/finite}{\texttt{family/certificates/finite}} contains the six PD
codes in \path{panel_pd.json}. From that directory, run
\texttt{python verify\_cases.py} to replay the three examples and check
the starting decompositions. The companion scripts check the intermediate knots and the three prime-factor names against table diagrams.

\subsection{Three changes for every twist length}\label{subsec:uniform}

Let $T_k$ be two strands with $k$ consecutive half-twists, of the handedness
in \autoref{fig:family-input}. The four marked points of its disc are
$0,3,2,1$ in clockwise order, starting at the upper left of the key.
Thus $T_0$ consists of two vertical strands. Write $D_k$ for the diagram
obtained by inserting $T_k$ into the marked disc.

\begin{figure}[htbp]\centering
  \includegraphics[width=.96\linewidth]{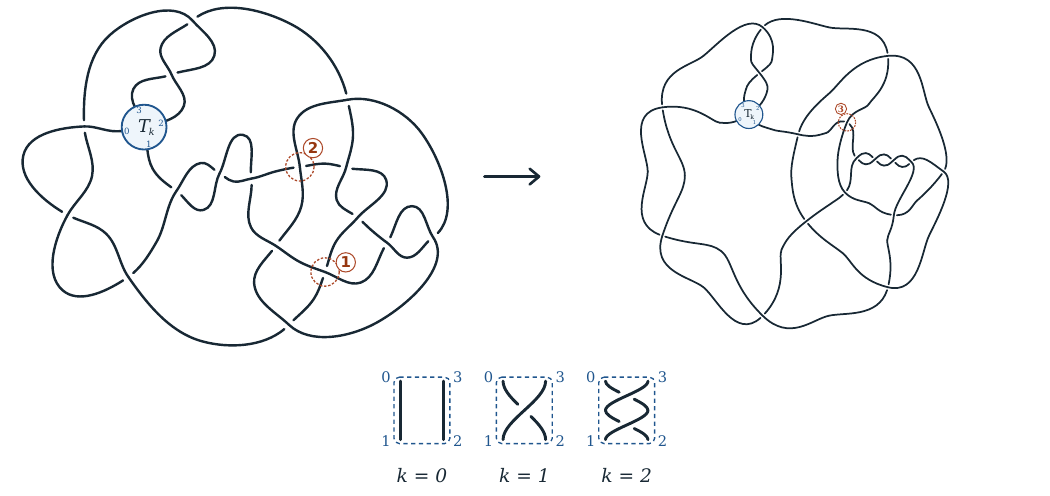}
  \caption{The same three changes for every $k$. In the left diagram change
    marks 1 and 2. Seven Reidemeister III moves give the right diagram,
    where we change mark 3. Match the four numbered attachment points to
    insert $T_k$ in either diagram. The bottom row shows its first three
    values.}\label{fig:family-input}
\end{figure}

\begin{Lemma}\label{lem:uniform-input}
Up to mirror image, $D_k$ is $4_1\#K_k$, where
$K_k$ is the two-bridge knot of fraction $[3,3,3,2,k]^+$.
\end{Lemma}
\begin{proof}
Three Reidemeister III moves and one Reidemeister II move outside the
disc expose a separating circle. The piece on one side is the four-crossing
figure-eight. The other side is the four-plat in \autoref{fig:family-plat}.
The row data below specify the moves, the separating circle, and the attachment
of the marked disc.

\begin{figure}[htbp]\centering
\includegraphics[width=.66\linewidth]{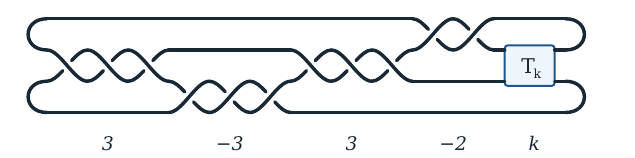}
\caption{The second summand after cutting. Read the five twist boxes from
left to right and number the four wires from the bottom.}
\label{fig:family-plat}
\end{figure}

Number the four wires from one to four, and let $b_i$ be the half-twist
in which the incoming left wire passes under its neighbor when the plat
is drawn vertically and read from top to bottom. The braid is
$b_2^3 b_1^{-3} b_2^3 b_3^{-2} b_2^k$, with consecutive pairs capped at
each end. To see why either outer pair gives the same
fraction, color the two top caps $0$ and $1$. Propagate the colors by
requiring the two under-colors at a crossing to sum to twice the over-color.
Across the four wires the colors $a,b,c,d$ satisfy $a-b+c-d=0$, and the
ratio of $X=b-a$ and $Y=c-b$ determines the rational tangle fraction
\cite{KauffmanLambropoulou-rational}. A left-under twist on the first
two wires gives $b,2b-a,c,d$. On the last two it gives $a,b,d,2d-c$.
Both send $(X,Y)$ to $(X,Y-X)$. Thus $b_3^{-2}$ gives the same fraction
as the usual outer box $b_1^{-2}$. The usual continued-fraction rule for the four-plat gives
$[3,3,3,2,k]^+=(76k+33)/(23k+10)$ up to the common mirror convention.
\end{proof}

\begin{Proposition}\label{prop:uniform-upper}
For every $k\geq0$, both $u(K_k)$ and $u(4_1\#K_k)$ are at most three.
\end{Proposition}
\begin{proof}
In the factor $K_k$, change two crossings of the second $3$-box and one
crossing of the third $3$-box. After cancelling opposite twists the fraction
is $[3,-1,1,2,k]^+=-1/k$. Its numerator closure is the unknot. At $k=0$
the fraction is $\infty$, whose numerator closure is again a single
unknotted circle.

For the sum, change marks 1 and 2 on the left of \autoref{fig:family-input}.
Seven triangle moves give the right-hand diagram. Change mark 3,
then simplify outside the marked disc. The only two arcs left outside
are the upper and lower caps in \autoref{fig:family-cap}. Slide the upper
cap past each half-twist in turn. Each slide removes one curl, leaving a
circle after $k$ slides. At $k=0$ the caps and the two straight strands
already form that circle.
\end{proof}

\begin{figure}[htbp]\centering
\includegraphics[width=.55\linewidth]{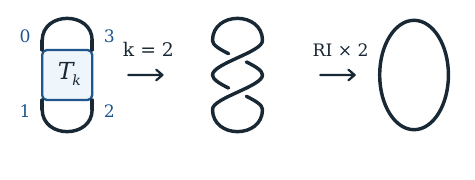}
\caption{After three changes the outside joins $0$ to $3$ and $1$ to $2$.
The two-crossing example shows how the upper cap removes the twists.}
\label{fig:family-cap}
\end{figure}

\smallskip\noindent\textit{Checking the marked diagrams.}
The following data allow a replay for all $k$ at once. Treat the marked
disc as a vertex with four numbered ports. Every move below takes place
on a face with ordinary crossings on its boundary, so fixes the disc.
Read the following four-tuples across the rows, starting at vertex zero.
At an ordinary vertex the first and third ports pass under. The last tuple
records the four ports of the disc in the order $0,1,2,3$.
\[
\begin{array}{@{}llll@{}}
(11,17,12,16)&(26,15,27,16)&(14,7,15,8)&(8,13,9,14)\\
(25,13,26,12)&(29,10,30,11)&(4,10,5,9)&(6,27,7,28)\\
(28,5,29,6)&(17,30,18,31)&(31,18,32,19)&(19,32,20,33)\\
(33,24,34,25)&(23,34,24,35)&(35,1,36,23)&(3,37,4,36)\\
(37,3,0,2)&\boxed{(1,20,2,0)}&&
\end{array}
\]
A triple of vertex labels specifies a triangle move, a pair a bigon
removal, and a singleton a curl removal. The labels stay attached to
their crossings. For the initial splitting the four moves are
$(4,0,1),(8,5,6),(4,6,7),(4,5)$.
Write $v_i$ for port $i$ of vertex $v$, starting at port zero.
Cut the two edges $0_1$--$1_1$ and $6_0$--$7_2$.
Join $0_1$ to $6_0$ and $1_1$ to $7_2$ on the respective sides.
Vertices $0,2,3,6$ give the standard $[2,1,1]^+=5/2$ figure-eight
plat. The five boxes on the other side have vertices, in braid order,
\[
 (8,7,1),\quad(9,10,11),\quad(12,13,14),\quad(15,16),
 \quad\boxed{17}.
\]
The last identification matches all four ports without rotation.

For the unknotting sequence, return to the original diagram and change
vertices $1$ and $5$. The seven triangle moves are
\[
 (5,0,9),\ (0,5,10),\ (5,0,11),\ (2,1,7),\
 (2,9,8),\ (10,2,8),\ (8,2,11).
\]
Change vertex $2$ of this diagram. The remaining moves, read across
the rows, are
\[
\begin{array}{@{}llllll@{}}
 (1,9)&(7,10,3)&(4,10)&(0,6,2)&(11,0,7)&(0,3)\\
 (5,11,12)&(12,8,7)&(12)&(6,8,7)&(6)&(11,5,13)\\
 (8,7,13)&(5,7)&(13)&(8,11,14)&(14)&(2,8,11)\\
 (11,15)&(2)&(8,16)&&&
\end{array}
\]
The final edges are $17_0$--$17_3$ and $17_1$--$17_2$.
The script \href{https://github.com/dtubbenhauer/upperbounds-reworked/blob/main/family/certificates/verify_uniform.py}{\texttt{verify\_uniform.py}} checks the faces, the two
marked plat identifications, and these final pairings.

\subsection{Why even indices require three changes}\label{subsec:family-lower}

An unknotting sequence gives an upper bound by displaying the crossings to
change. For even $k$, to rule out \emph{every} two-change sequence for $K_k$ and
$4_1\#K_k$, we count coordinates in an integer lattice. We first give the
counting rule, then read off the numbers from the rational knot diagram.

Write $\Sigma(J)$ for the double cover of $S^3$ branched along a knot $J$.
A \emph{cap} here is a smooth compact oriented four-manifold $P$ with
boundary $-\Sigma(J)$. Its intersection form records the intersection numbers
of pairs of two-dimensional homology classes. We use caps with positive
definite form. The \emph{square} of a class $x$ means its self-intersection
$x\cdot x$.

\begin{Lemma}\label{lem:family-coordinate-count}
Suppose $J$ has signature $4$ and a simply connected positive definite cap with an intersection
matrix of rank $r$ and odd determinant. In a basis for this matrix, select
$s$ vectors of square two. Join two of them when their pairing is nonzero.
If this graph is a union of $c$ paths, counting an isolated vertex as one
path, then
\[
                       u(J)=2\quad\Longrightarrow\quad s+c\leq r.
\]
\end{Lemma}
\begin{proof}
We use the signature convention in which changing a negative crossing can
lower the signature by two. To get from $4$ to $0$ in two changes, both
changes must therefore be at negative crossings. Owens's surgery theorem
\cite[Theorem~3 and Lemma~2.2]{Owens-HF-unknotting} gives a simply connected
positive definite four-manifold $W$ with boundary $\Sigma(J)$. Its intersection
form has rank four, determinant $\det(J)$, and two orthogonal basis vectors
of square two. They are the two meridians of framing $2$ in the integral surgery diagram.

Glue $W$ to the cap along their boundaries. The result is closed and simply
connected. Since $\Sigma(J)$ is a rational homology sphere, the two forms
give orthogonal sublattices whose ranks add to $r+4$. Donaldson's
diagonalization theorem \cite{Donaldson-diagonal} identifies the form of
the glued manifold with the usual dot product on $\Z^{r+4}$. Let $M$ be the coordinate matrix of the two bases together. Since
$(\det M)^2$ is the product of two odd form determinants, $M$ is invertible
modulo two.

A vector of square two in the standard lattice has two nonzero coordinates,
both $\pm1$. Draw an edge between their coordinate positions. No two
selected vectors have the same edge, since then their reductions modulo two
would coincide. Two vectors pair nontrivially precisely when their edges
share an endpoint. For example, two neighboring vectors and a vector
orthogonal to both use the five coordinates shown below.
\[
\begin{tikzpicture}[x=.73cm,y=.54cm,baseline=(current bounding box.center)]
  \draw[thick] (0,0)--(1,0)--(2,0) (3.35,0)--(4.35,0);
  \foreach \x/\name in {0/i,1/j,2/\ell,3.35/a,4.35/b} {
    \fill (\x,0) circle (2pt);
    \node[below,font=\small] at (\x,-.06) {$\name$}; }
  \node[above,font=\small] at (.5,.04) {$x_1$};
  \node[above,font=\small] at (1.5,.04) {$x_2$};
  \node[above,font=\small] at (3.85,.04) {$x_3$};
\end{tikzpicture}
\]
After the first edge in a path, each new edge meets its predecessor and
misses all earlier edges. Thus a path of $d$ vectors uses $d+1$ coordinates.
Different paths use disjoint coordinates. The two square-two vectors from
$W$ use four further coordinates, so $s+c+4\leq r+4$.
\end{proof}

\begin{Proposition}\label{prop:family-lower}
For every even $k\geq0$, both $K_k$ and $4_1\#K_k$ have unknotting number
at least three.
\end{Proposition}
\begin{proof}
Put $p=76k+33$ and $q=23k+10$, so $K_k$ has two-bridge fraction $p/q$.
We write $[a_1,\ldots,a_m]^-=a_1-1/[a_2,\ldots,a_m]^-$, ending with
$[a_m]^-=a_m$. A linear plumbing with weights $a_i$ has intersection
matrix with these diagonal entries, $-1$ between neighbors and zeros
elsewhere. Its boundary is $[a_1,\ldots,a_m]^-$ surgery on the unknot.
We choose the mirror convention in which $\Sigma(K_k)$ is $+p/q$ surgery.
Mirroring preserves unknotting number, and $4_1$ is amphichiral.

For even $k\geq2$, $p/q=[4,2,2,4,-2,k]^-$. This all-even chain is the
double branched cover of the orientable surface obtained by plumbing the
corresponding twisted annuli and pushing the interior into the four-ball.
Thus its matrix is the symmetrized Seifert form of this chirality
\cite[Theorem~3 and Corollary~5]{GordonLitherland-signature}. Its leading principal minors
are $4,7,10,33,-76,-p$. Successive ratios, starting with $4/1$, give five
positive pivots and one negative pivot. For $k=0$, the shorter chain $[4,2,2,4]^-=33/10$
has four positive pivots. Hence $\sigma(K_k)=4$ in both cases. The
figure-eight knot has signature zero, so the connected sum also has
signature four.

Reversing the orientation of $+p/q$ surgery gives the lens space presented
by $+p/(p-q)$ surgery since the denominators $-q$ and $p-q$ agree modulo $p$.
Expanding this second fraction gives the positive definite cap
\[
 \frac p{p-q}=
 \begin{cases}
 [2,2,5,2,2]^-,& k=0,\\[2pt]
 [2,2,5,2,2,4,\underbrace{2,\ldots,2}_{k-1}]^-,& k\geq2.
 \end{cases}
\]
The plumbing is simply connected, has determinant $p$, and its weights
are all at least two. For the connected sum we take its boundary connected
sum with the $[3,2]^-=5/2$ plumbing for the figure-eight cover. The
intersection matrices form a direct sum. Here is the cap at $k=2$, with the
square-two vectors shaded.
\[
\begin{tikzpicture}[x=.62cm,y=.5cm,baseline=(current bounding box.center),
   weight/.style={circle,draw=black,fill=white,minimum size=5.2mm,inner sep=0pt,font=\small},
   two/.style={weight,fill=black!16}]
  \draw (0,0)--(6,0) (8,0)--(9,0);
  \foreach \x/\w/\sty in {0/2/two,1/2/two,2/5/weight,3/2/two,4/2/two,5/4/weight,6/2/two,8/3/weight,9/2/two}
       \node[\sty] at (\x,0) {$\w$};
  \node[below,font=\small] at (3,-.55) {$K_2$};
  \node[below,font=\small] at (8.5,-.55) {$4_1$};
\end{tikzpicture}
\]
For $k=0$ the cap for $K_0$ has rank five. Its two runs of two shaded vertices
give $s+c=4+2=6$. For $k\geq2$ the cap for $K_k$ has rank $k+5$. Its three
runs have lengths $2,2,k-1$, giving $s+c=(k+3)+3=k+6$.
In either case $s+c=r+1$. Adding the figure-eight cap increases both $r$
and $s+c$ by two. Thus \autoref{lem:family-coordinate-count} excludes two
crossing changes for each knot, and signature four already excludes fewer.
\end{proof}

This completes the proof of \autoref{thm:figure-eight-family}.

\begin{Remark}[Montesinos partners]
The agent found one more three-change figure-eight sum. Varying a twist
box in its Montesinos partner made the family easy to guess. Write
$M(x,y,z)$ for the numerator closure of $[x]+[y]+[z]$, with horizontal
tangle sum. For even $k\geq2$, put $p=76k+33$ and
$J_k=M(1/2,-1/3,76/p)$. Then
\[
 u(F\#J_k)=u(J_k)=3<4=u(F)+u(J_k).
\]

Use the cap chart in \autoref{fig:family-cap}. Testing the rational
partner exterior on $[1/j]$ gives determinant $76j+33$, so its fraction
is $76/33$. The three changes of \autoref{prop:uniform-upper} leave
numerator closure, for the sum by the marked moves and for the partner
by the same test with determinant $1$. Replace the straight strands by
$[1/2]+[-1/3]$, keeping the $k$ cap twists. Moving them to the initial
exterior gives $76/p$. They slide off the final cap, leaving a two-bridge
unknot since $|1\cdot3-2\cdot1|=1$.

The table gives two plumbings with the same oriented double branched cover of $J_k$ as boundary.
\[
\begin{array}{@{}c l@{}}
\text{center}&\text{arms, read outwards}\\
0&(2),\ (-2,-2),\ (-k,2,-4,-2,-2,-4)\\
2&(2),\ (3),\ (\underbrace{2,\ldots,2}_{k-1},4,2,2,5,2,2).
\end{array}
\]
The central pivot is $1/6+76/p$ in both. The first has even weights and
signature $-4$, and the second is a positive cap for the signature-four
mirror. In the proof of \autoref{lem:family-coordinate-count} the
square-two vectors now use all $k+12$ coordinates in blocks
$k+2,3,3,2,2$. The weight-three vertex is $(a,a,b,\ldots,b)$ on the
first block with $|a-b|=1$, so that block contributes two to its square.
The remaining unit cannot fit on blocks where the entries are constant.
Adding the figure-eight cap adds one size-two block and gives the same
obstruction. The covers are irreducible Seifert manifolds over
$S^2(2,3,p)$ with noncyclic fundamental group, so the $J_k$ are prime
and non-two-bridge. Their determinants $p+456$ distinguish them.
\end{Remark}

\section{Upper bounds for prime knots}\label{sec:upper bounds}

For the following fifteen witnesses we tested every crossing of each chosen prime knot diagram. Start with a knot $K$ whose unknotting number lies in an interval $[\ell,u]$. A checked crossing change from $K$ to $L$ gives $u(K)\leq1+u(L)$. Once entered in the table, an improved upper bound for $K$ is available for later searches that reach $K$.

\subsection{Identifying the knot after a crossing change}

We simplify the changed diagram and filter possible targets by the Jones and Alexander polynomials, hyperbolic volume, and the polynomial $V_2$ of Garoufalidis and Kashaev \cite{GaroufalidisKashaev-Vn,GaroufalidisLi-V2}. For our invariant comparisons and normalization conventions see \cite{TubbenhauerZhang-bigdata-quantum-invariants,KelomakiLacabanneTubbenhauerVazZhang-detection-probabilities} (also using methods from \cite{LaTuVa-big-data}). Several knots can share the same filter values. For the witnesses below we further checked that the exteriors are isometric by a meridian-preserving map, allowing mirroring; the bound for the target must also be known. The notebook and the diagram data are available at \cite{DKT}.

\subsection{Fifteen upper bounds}\label{subsec:upper bounds-placeholder}

The run gives fifteen checked crossing change witnesses, listed in \autoref{tab:upper-bound-current}. Each witness leads to a knot with unknotting number one and hence gives an upper bound of two. The old intervals refer to the input snapshot used for the run (15 March 2026), rather than to the latest version of KnotInfo. In particular, we retain $12a{262}$ and $13n{489}$, whose unknotting number two is already recorded in KnotInfo as of 8 September 2026.

\begin{table}[ht]
\centering
\begin{tabular}{llll}
\toprule
Knot & input & after our run & knot after one change \\
\midrule
$12a{262}$ & $[2,3]$ & $2$ & $6_{3}$ \\
$13n{45}$ & $[2,3]$ & $2$ & $10_{23}$ \\
$13n{80}$ & $[2,3]$ & $2$ & $10_{95}$ \\
$13n{221}$ & $[1,3]$ & $[1,2]$ & $10_{84}$ \\
$13n{489}$ & $[1,3]$ & $2$ & $10_{131}$ \\
$13n{636}$ & $[1,3]$ & $[1,2]$ & $10_{91}$ \\
$13n{1439}$ & $[1,3]$ & $[1,2]$ & $9_{27}$ \\
$13n{2251}$ & $[1,3]$ & $[1,2]$ & $11a{77}$ \\
$13n{2379}$ & $[2,3]$ & $2$ & $11a{345}$ \\
$13n{2639}$ & $[1,3]$ & $[1,2]$ & $11a{146}$ \\
$13n{2809}$ & $[2,3]$ & $2$ & $9_{22}$ \\
$13n{2907}$ & $[2,3]$ & $2$ & $10_{133}$ \\
$13n{3033}$ & $[2,3]$ & $2$ & $12n{646}$ \\
$13n{3589}$ & $[2,3]$ & $2$ & $10_{131}$ \\
$13n{4025}$ & $[1,3]$ & $[1,2]$ & $10_{26}$ \\
\bottomrule
\end{tabular}
\caption{Checked upper bounds relative to the input snapshot. The third column incorporates the lower bounds available at the time of the run. For $13n{489}$ we use the updated lower bound two. Knot names in the last column are understood up to mirroring.}
\label{tab:upper-bound-current}
\end{table}

Together with the lower bounds used for this run, the witnesses gave nine exact values and six intervals $[1,2]$. Lee subsequently gave marked diagrams for thirteen of the fifteen knots and lower bounds of two for all six remaining intervals \cite[Section~4.3]{Lee-BJ}. Thus every knot in \autoref{tab:upper-bound-current} has unknotting number two. The identifications were checked using SnapPy isometries that preserve meridians, both for the source diagrams and for the diagrams after the crossing change. These are numerical isometry checks, not interval-certified computations. The target upper bounds can also be checked by explicit single crossing changes followed by simplification to the unknot. All fifteen witnesses were found in the preliminary simplification stage, before the RL reducer was called.

Let us study $K=13n{2639}$ in detail. Its starting diagram $D$ has the following thirteen-crossing PD code:
\begin{gather*}
\begin{gathered}
\relax [[3,1,4,26],[1,9,2,8],[9,3,10,2],[4,18,5,17],\\
\relax [18,6,19,5],[6,15,7,16],[14,7,15,8],[10,21,11,22],\\
\relax [22,11,23,12],[12,23,13,24],[24,13,25,14],[16,20,17,19],\\
\relax [20,25,21,26]],
\end{gathered}
\\
\vcenter{\hbox{\includegraphics[height=4cm]{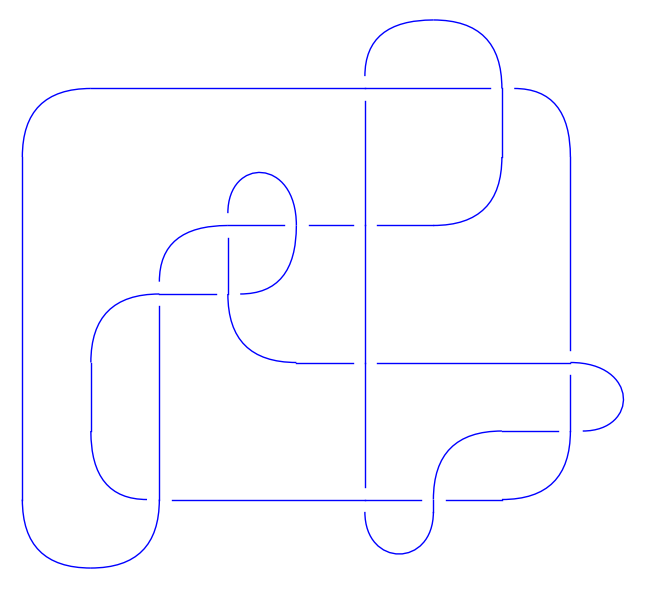}}}
.
\end{gather*}
We first check why an increase in crossings is needed for this starting diagram. Changing each of its thirteen crossings gives the following knots, up to mirroring. Here and below, crossing indices start at zero and refer to the displayed order of the quadruples.
\begin{center}
\begin{tabular}{lll}
\toprule
crossing indices in $D$ & resulting knot & unknotting number \\
\midrule
$0$ & $12a{649}$ & $[2,3]$ \\
$1,2$ & $10_{12}$ & $2$ \\
$3,4$ & $10_{126}$ & $2$ \\
$5,6$ & $10_{48}$ & $2$ \\
$7,8,9,10$ & $11n{106}$ & $2$ \\
$11$ & a nontrivial connected sum & $\geq2$ \\
$12$ & $11a{109}$ & $2$ \\
\bottomrule
\end{tabular}
\end{center}
The prime knot bounds in this table are recorded in KnotInfo \cite{KnotInfo}. At index $11$, the diagram splits into summands with determinants $3$ and $5$, so both summands are nontrivial. Its unknotting number is therefore at least two by Scharlemann's theorem that unknotting-number-one-knots are prime. See \cite{Scharlemann-prime}.

Thus no crossing change on $D$ leads to a knot with unknotting number at most one. Moreover, $D$ has no available Reidemeister III move. Since $D$ is a minimal diagram, no decreasing Reidemeister I or II move is available either. Consequently, a Reidemeister-move-search starting from $D$ must first increase the crossing number before it can expose an improving first crossing change.

The saved witness uses the following twenty-four crossing diagram $D'$ of the same knot:
\begin{gather*}
\begin{gathered}
\relax [[6,44,7,43],[0,18,1,17],[20,2,21,1],[7,35,8,34],\\
\relax [35,11,36,10],[13,32,14,33],[29,16,30,17],[25,38,26,39],\\
\relax [39,26,40,27],[27,40,28,41],[41,28,42,29],[33,37,34,36],\\
\relax [37,42,38,43],[47,12,0,13],[46,12,47,11],[8,10,9,9],\\
\relax [21,2,22,3],[24,6,25,5],[44,20,45,19],[45,18,46,19],\\
\relax [3,22,4,23],[4,24,5,23],[14,32,15,31],[15,30,16,31]],
\end{gathered}
\\
\raisebox{-3cm}{\includegraphics[height=5.5cm]{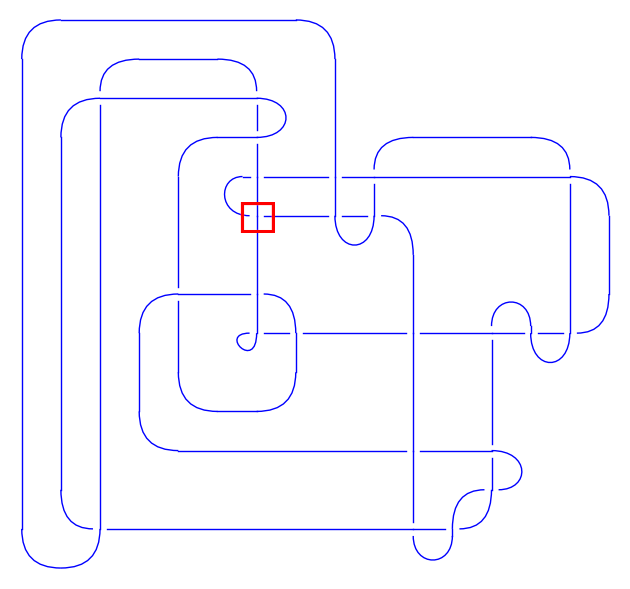}}
.
\end{gather*}
Change crossing $14$ (marked above), that is, the fifteenth quadruple, by replacing
\[
\relax [46,12,47,11]\quad\longmapsto\quad[11,46,12,47].
\]
Every other quadruple stays unchanged. 

\begin{Remark}\label{R:Stuck3}
Before the change, crossings $13$ and $14$ share arcs $12$ and $47$ and form a removable Reidemeister II pair. Changing crossing $14$ prevents this cancellation.
\end{Remark}

After simplification, the resulting diagram has PD code
\begin{gather*}
\begin{gathered}
\relax [[0,8,1,7],[2,18,3,17],[18,4,19,3],[5,14,6,15],\\
\relax [13,6,14,7],[9,20,10,21],[21,10,22,11],[11,22,12,23],\\
\relax [23,12,24,13],[16,20,17,19],[25,4,0,5],[8,2,9,1],\\
\relax [15,24,16,25]],
\end{gathered}
\\
\raisebox{-3cm}{\includegraphics[height=5.5cm]{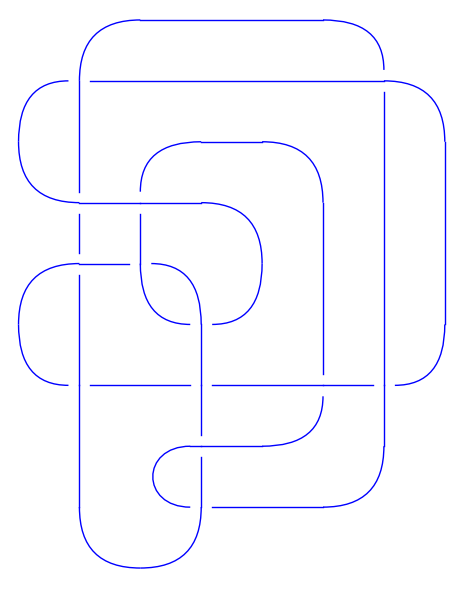}}
.
\end{gather*}
The preceding diagram represents $11a{146}$ up to mirroring, as checked by a meridian-preserving SnapPy isometry. Since $u(11a{146})=1$, we obtain
\[
u(13n{2639})\leq 1+u(11a{146})=2,
\]
improving the input interval $[1,3]$ to $[1,2]$. Lee's calculation gives the matching lower bound.

For completeness, a standard eleven-crossing diagram of the target is
\begin{gather*}
\begin{gathered}
\relax [[2,0,3,21],[0,9,1,10],[8,1,9,2],[14,3,15,4],\\
\relax [4,15,5,16],[16,5,17,6],[6,19,7,20],[12,8,13,7],\\
\relax [10,18,11,17],[18,12,19,11],[20,13,21,14]],
\end{gathered}
\\
\raisebox{-3cm}{\includegraphics[height=5.5cm]{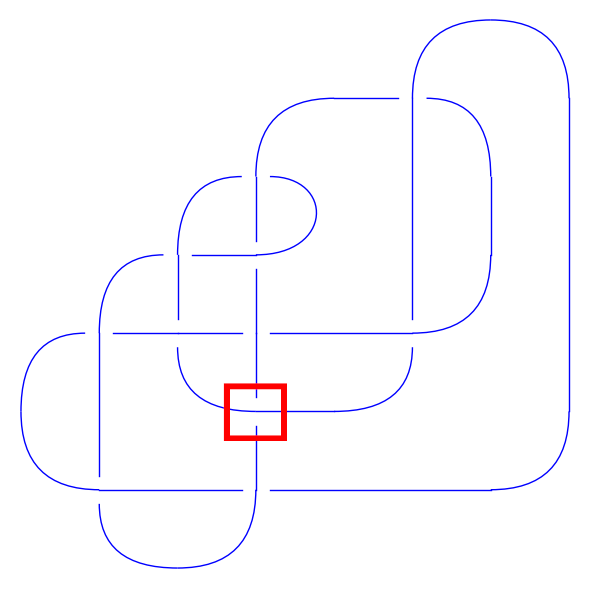}}.
\end{gather*}
Changing its crossing index $6$, that is, the seventh crossing (marked above), namely replacing $[6,19,7,20]$ by $[19,7,20,6]$, gives a diagram that simplifies to the unknot.

\begin{Remark}
The need to increase crossings is established for the particular starting diagram $D$. Lee gives a different thirteen-crossing diagram of $13n{2639}$ with two marked crossings whose changes unknot it \cite[Figure~5]{Lee-BJ}. The example shows how inflation makes an improving crossing change accessible from $D$.
\end{Remark}

\end{document}